\documentclass[a4paper,10pt]{article}
\usepackage{amsmath,amssymb} 
\usepackage{amsthm}

\usepackage{graphicx}
\usepackage{pstricks-add}
\usepackage{subfigure}
\usepackage{psfrag}

\newtheorem{theorem}{Theorem}
\newtheorem{proposition}{Proposition}

\newtheorem{corollary}{Corollary}
\newtheorem{hypotheses}{Hypotheses}
\newtheorem{remark}{Remark}

\numberwithin{figure}{section}
\numberwithin{table}{section}
\numberwithin{theorem}{section}
\def\m{{\rm m}}
\def\ds{\displaystyle}
\def\pa{\partial}
\def\Div{\mathrm{div}}

\def\N{\mathbb N}
\def\R{\mathbb R}
\def\T{\mathcal T}
\def\mK{{\rm m}(K)}

\def\FKsig{{\mathcal F}_{K,\sigma}}

\def\GKsig{{\mathcal G}_{K,\sigma}}
\def\DPsiKs{D\Psi_{K,\sigma}}
\def\nKsig {\nu_{K,\sigma}}

\def\ts{\tau_{\sigma}}

\def\E{\mathcal E}

\def\somK{\ds\sum_{K\in\T}}
\def\somsigint{\ds\sum_{\substack{\sigma \in \E_{int}\\ \sigma=K|L}}}
\def\somn{\ds\sum_{n=0}^{N_T-1}}
\def\somsig{\ds\sum_{\substack{\sigma \in \E\\ (K=K_{\sigma})}}}
\def\somsige{\ds\sum_{\sigma \in \E}}

\def\NKn{N_K^n}
\def\PKn{P_K^n}
\def\NKnp{N_K^{n+1}}
\def\PKnp{P_K^{n+1}}

\def\Pl{(\cal P_{\lambda})}
\def\Pz{({\cal P}_{0})}
\def\Sl{(\cal S_{\lambda})}
\def\Sz{({\cal S}_{0})}

\title{Study of a finite volume scheme for the drift-diffusion system. Asymptotic behavior in the quasi-neutral limit.}

\author{M. Bessemoulin-Chatard,\thanks{Universit\'e de Nantes, Laboratoire de Math\'ematiques Jean Leray, UMR 6629 - CNRS, 2 rue de la Houssini\`ere - BP 92208,  44322 Nantes Cedex 3, France
 ({\tt Marianne.Bessemoulin@univ-nantes.fr}).}
        \and C. Chainais-Hillairet, \thanks{Laboratoire P. Painlev\'e, CNRS UMR 8524, Universit\'e Lille 1, 59655 Villeneuve d'Ascq Cedex. Project-Team MEPHYSTO, INRIA Lille Nord Europe, 40 av. Halley, 59650 Villeneuve d'Ascq, France. 
         ({\tt Claire.Chainais@math.univ-lille1.fr}).}
        \and M.-H. Vignal\thanks{Institut de Math\'ematiques de Toulouse, 
CNRS UMR 5219, Universit\'e de Toulouse, 
118 route de Narbonne,  31062 TOULOUSE cedex 9, 
France ({\tt mhvignal@math.univ-toulouse.fr}).} }
\date{ }

\begin{document}

\maketitle

\begin{abstract}
In this paper, we are interested in the numerical approximation of the classical time-dependent drift-diffusion system near quasi-neutrality. We consider a fully implicit in time and finite volume in space scheme, where the convection-diffusion fluxes are approximated by Scharfetter-Gummel fluxes. We establish that all the a priori estimates needed to prove the convergence of the scheme does not depend on the Debye length $\lambda$.  This proves that the scheme  is asymptotic preserving in the quasi-neutral limit $\lambda \to 0$. 
\end{abstract}



\section{Introduction}\label{sec.intro}

\subsection{Aim of the paper}

In the modeling of plasmas or semiconductor devices, there is a hierarchy of different models: kinetic models and quasi hydrodynamic models, ranging from Euler-Poisson system to drift-diffusion systems (see \cite{Markowich1986,MRS,Jungelbook,Jungelbook2}).  In each of these models scaled parameters are involved, like the effective mass of electrons, the relaxation time or the rescaled Debye length.  
There is a wide literature on the theoretical validation of the hierarchy of models (see \cite{NBAD,JungelPeng2001,CDS} and references therein). Moreover, an active and recent field of research consists in designing numerical schemes for these physical models which are valid for all range of scaled parameters, and especially when these parameters may tend to 0. These schemes are said to be asymptotic preserving. These methods have proved their efficiency in many situations, for instance: in fluid limits for the Vlasov equation, quasi-neutral limits for the drift-diffusion, Euler or Vlasov equations coupled to the Poisson equation, in diffusive limit for radiative transfer (see \cite{Jin,CDV,Lemoumieussens,BCDS,DDNSV,FilbetJin,ChainaisVignal} among a long list of articles that could not be mentioned here)

In this paper, we consider the numerical approximation of the linear drift-diffusion system. It is a coupled system of parabolic and elliptic equations involving only one dimensionless parameter: $\lambda$, the rescaled Debye length. This parameter $\lambda$ is given by the ratio of the Debye length to the size of the domain; it  measures the typical scale of electric interactions in the semiconductor. Many different numerical methods have been already developed for the approximation of the drift-diffusion system; 
see for instance the mixed exponential fitting schemes proposed in \cite{BMP}  and extended in \cite{Jungel1995ZAMM,JungelPietra}
to the case of nonlinear diffusion. The convergence of some finite volume schemes has been proved by C. Chainais-Hillairet, J.-G. Liu and Y.-J. Peng  in \cite{CCHLP,CCHP}. But, up to our knowledge, all the schemes are studied in the case $\lambda=1$ and the behavior when $\lambda$ tends to 0 has not yet been studied. 

In this paper, we are interested in designing and studying a scheme for the drift-diffusion system applicable for any value of $\lambda$. This scheme must converge for any value of $\lambda\geq 0$ and must remain stable at the quasi-neutral limit $\lambda\to 0$. We consider an implicit in time  and finite volume in space  scheme  with a Scharfetter-Gummel approximation of the convection-diffusion fluxes. As it is classical in the finite volume framework (see \cite{EGHbook}), the proof of convergence of the scheme is based on some a priori estimates which yield the compactness of the sequence of approximate solutions.  In the case of the drift-diffusion system, the a priori estimates needed for the proof of convergence are $L^{\infty}$ estimates on $N$ and $P$, discrete $L^2(0,T,H^1(\Omega))$-estimates on $N$, $P$ and $\Psi$ in the non-degenerate case \cite{CCHLP}, with additional weak-BV estimates on $N$ and $P$ in the degenerate case \cite{CCHP}. 
However, the crucial point in our work is to establish that all the a priori estimates do not depend of $\lambda\geq 0$ and therefore the strategy used in \cite{CCHLP,CCHP} to get them does not directly apply. In order to get estimates which are independent of $\lambda$, we adapt to the discrete level the entropy method   proposed by A.~J\"ungel and Y.-J.~Peng in \cite{JungelPeng2001} and by I. Gasser {\em et al} in \cite{Gasser, GLMS}. The choice of the Scharfetter-Gummel fluxes for the discretization of the convection-diffusion fluxes is essential at this step.

\subsection{The drift-diffusion system}

Let $\Omega$ be an open bounded subset of $\R^d$ ($d\geq 1$)  describing the geometry of a semiconductor device and $T>0$. The unknowns are the density of electrons  and holes, $N$ and $P$,  and the electrostatic potential $\Psi$. This device can be described  by the so-called drift-diffusion system introduced by W. Van Roosbroeck \cite{VanRoosbroeck} (see also \cite{Gajewski_Groger,Markowich1986,MRS}). It writes for all $(x,t)\in\Omega\times [0,T]$:
\begin{subequations}\label{DD-complet}
\begin{equation}
\pa_{t}N+\Div\left(\mu_{N}(-\nabla N+N\nabla \Psi)\right)=-R(N,P),
\end{equation}
\vspace*{-0.5cm}
\begin{equation}
\pa_{t}P+\Div\left(\mu_{P}(-\nabla P-P\nabla \Psi)\right)=-R(N,P),
\end{equation}
\vspace*{-0.5cm}
\begin{equation}
-\lambda^{2}\Delta \Psi=P-N+C,
\end{equation}
\end{subequations}
where the given function $C(x)$ is the doping profile describing fixed background charges and $R(N,P)$ is the recombination-generation rate, which is usually taken as the Shockley-Read-Hall term:
$$ R(N,P)=\frac{NP-N_{i}^{2}}{\tau_{P}N+\tau_{N}P+\tau_{C}}, \quad \tau_{P},\, \tau_{N},\, \tau_{C},\, N_{i}>0.$$
The dimensionless physical parameters $\mu_{N}$, $\mu_{P}$ and $\lambda$ are the rescaled mobilities of electrons and holes, and the rescaled Debye length respectively. 
The system is supplemented with mixed boundary conditions (see \cite{Markowich1986}):  Dirichlet boundary conditions on the ohmic contacts and homogeneous boundary conditions on the insulated boundary segments. It means that the boundary $\partial\Omega$ is split into $\partial \Omega= \Gamma^D \cup \Gamma^N$ with $\Gamma^D\cap \Gamma^N=\emptyset$ and that the boundary conditions write:
\begin{subequations}\label{CL}
\begin{equation}\label{CL_dir}
N(\gamma,t)=N^D(\gamma),\   P(\gamma,t)=P^D(\gamma), \Psi(\gamma,t)=\Psi^D(\gamma), \ (\gamma,t)\in\Gamma^D\times[0,T],
\end{equation}
\vspace{-0.5cm}
\begin{equation}\label{CL_neum}
(\nabla N\cdot\nu)\ (\gamma,t) =(\nabla P\cdot\nu)\ (\gamma,t)=(\nabla \Psi\cdot\nu)\ (\gamma,t)=0 , \ (\gamma,t)\in\Gamma^N\times [0,T],
\end{equation}
\end{subequations}
where $\nu$ is the unit normal to $\partial\Omega$ outward to $\Omega$.

The system \eqref{DD-complet} is also supplemented with initial conditions $N_0$, $P_0$:
\begin{equation}\label{CI} 
N(x,0)=N_0(x),\quad P(x,0)=P_0(x) ,\quad x\in\Omega.
\end{equation}

In this paper, we want to focus on the stability of some numerical schemes for the drift-diffusion system with respect to the Debye length $\lambda$. Therefore, as in the theoretical analysis \cite{JungelPeng2001,Gasser, GLMS}, we will consider a simplified model: we neglect the recombination-generation rate $R$, we assume that the mobilities are constant and equal $\mu_{N}=\mu_{P}=1$ and we also assume that the doping profile $C$ vanishes. 
However,  in Section \ref{sec.numexp}, we will provide some numerical experiments with vanishing and non-vanishing doping profiles. 

 In the sequel, we denote by $(\cal P_{\lambda})$ the simplified model under study in this paper, given by:
\begin{subequations}\label{DD}
\begin{gather}
\pa_t N+\Div (-\nabla N+N\nabla\Psi)=0,\label{DD-N}\\
\pa_t P+\Div (-\nabla P-P\nabla\Psi)=0,\label{DD-P}\\
-\lambda^2 \Delta\Psi=P-N,\label{DD-Psi}
\end{gather}
\end{subequations}
supplemented with boundary conditions \eqref{CL} and initial conditions \eqref{CI}. We need the following assumptions:
\begin{hypotheses}\label{HYP}
The domain $\Omega$ is an open bounded subset of $\R^d$ ($d\geq 1$) and $\pa \Omega= \Gamma^D\cup\Gamma^N$ with $\Gamma^D\cap\Gamma^N=\emptyset$ and ${\rm m}(\Gamma^D)>0$.  The boundary conditions $N^D$, $P^D$ and $\Psi^D$ are the traces of some functions defined on the whole domain $\Omega$, still denoted by $N^D$, $P^D$ and $\Psi^D$. 
Furthermore, we assume that 
\begin{subequations}\label{HYP_data}
\begin{gather}
N_0,P_0\in L^{\infty}(\Omega),\label{HYP-CI}\\
N^D,P^D\in L^{\infty}\cap H^1(\Omega),\  \Psi^D\in H^1(\Omega),\label{HYP-CL}\\
\exists m>0, M >0 \mbox{ such that } m\leq N_0,P_0,N^D,P^D\leq M  \mbox{ a.e. on } \Omega.\label{HYP-mM}
\end{gather}
\end{subequations}

\end{hypotheses}

The weak solution of $(\cal P_{\lambda})$ is defined by: 
$N$, $P \in L^{\infty}(\Omega \times (0,T))$,
$N-N^{D}$, $P-P^{D}$, $\Psi-\Psi^{D} \in L^{\infty}(0,T;V)$,
with $V=\bigl\{v\in H^1(\Omega)\;;\; v=0\mbox{ almost everywhere on } \Gamma^D\bigl\}$
and, for all test functions $\varphi \in C^{\infty}_c(\overline\Omega \times [0,T))$ and $\eta \in C^{\infty}_c(\overline\Omega \times (0,T))$ such that $\varphi(\gamma,t)=\eta(\gamma,t)=0$ for all  $(\gamma,t)\in\Gamma^D\times [0,T)$:
\begin{subequations}
\begin{gather}
\int_{0}^{T}\int_{\Omega}(N\,\pa_{t}\varphi-\nabla N \cdot \nabla \varphi+ N\,\nabla\Psi \cdot \nabla \varphi)\,dx \,dt+\int_{\Omega}N_{0}(x)\,\varphi(x,0)\,dx=0,\label{weaksol.N}\\
\int_{0}^{T}\int_{\Omega}(P\,\pa_{t}\varphi-\nabla P \cdot \nabla \varphi- P\,\nabla\Psi \cdot \nabla \varphi)\,dx \,dt+\int_{\Omega}P_{0}(x)\,\varphi(x,0)\,dx=0,\label{weaksol.P}\\
\lambda^{2} \int_{0}^{T}\int_{\Omega}\nabla \Psi \cdot \nabla \eta\,dx\,dt=\int_{0}^{T}\int_{\Omega}(P-N)\,\eta \,dx\,dt.\label{weaksol.Psi}
\end{gather}
\end{subequations}
The existence of a weak solution to the drift-diffusion system $(\cal P_{\lambda})$ has been proved in \cite{Gajewski1985,Mock}
under hypotheses more restrictive than Hypotheses \ref{HYP} since they consider more regular boundary conditions.
In \cite{Gajewski_Groger}, the authors prove these existence results under Hypotheses \ref{HYP} and assuming that
$\nabla (\log N_D -\Psi_D)$, $\nabla (\log P_D +\Psi_D)$ are in $L^{\infty}(\Omega)$.  

\subsection{The quasi-neutral limit of the drift-diffusion system}

The quasi-neutral limit plays an important role in many physical situations like sheath problems \cite{Franklin}, plasma diode modeling \cite{Sze_etal},  semiconductors \cite{Sze},... Then, it has been studied for different models: see \cite{cordier_grenier,slemrod} for the Euler-Poisson model, \cite{brenier_grenier,grenier} for the Vlasov-Poisson model and \cite{JungelPeng2001,Gasser, GLMS} for the 
drift diffusion-Poisson model.

In these models, the quasi-neutral limit consists in letting the scaled Debye length $\lambda$ tending to zero. From a physical point of view, this means that only the large scale structures with respect to the Debye length are then taken into account. 
Formally, this quasi-neutral limit is obtained by setting $\lambda=0$ in the model, here $\Pl$. 
Then, the Poisson equation \eqref{DD-Psi} on $\Psi$  reduces to the algebraic relation $P-N=0$ (which is the quasi-neutrality relation). But adding and subtracting \eqref{DD-N} and \eqref{DD-P}, we get  new equations on $N$ and $\Psi$. The quasi-neutral system $\Pz$ rewrites finally for all $(x,t)\in \Omega\times [0,T]$:
\begin{subequations}\label{QNC0}
\begin{gather}
\pa_t N -\Delta N =0,\label{QNC0-N}\\
{\rm div}(N\nabla \Psi )=0\label{QNC0-Psi},\\
P=N.\label{QN-P}
\end{gather}
\end{subequations}

In \cite{JungelPeng2001}, A. J\"{u}ngel and Y.-J. Peng performed rigorously the quasi-neutral limit for the drift-diffusion system with a zero doping profile and mixed Dirichlet and homogeneous Neumann boundary conditions. Indeed, under Hypotheses~\ref{HYP} and under quasi-neutrality assumptions on the initial and boundary conditions ($N_0-P_0=0$ and $N^D-P^D=0$), they prove that a weak solution to $\Pl$, denoted by $(N^{\lambda}, P^{\lambda}, \Psi^{\lambda})$, converges, when $\lambda\to 0$, to $(N^0,P^0,\Psi^0)$ solution to $\Pz$ in the following sense:
$$
\begin{gathered}
N^{\lambda} \to N^0,  P^{\lambda} \to P^0 \mbox{ in } L^p(\Omega\times (0,T)) \mbox{ strongly, for all } p\in [1,+\infty),\\
N^{\lambda} \rightharpoonup N^0,  P^{\lambda} \rightharpoonup P^0 ,  \Psi^{\lambda} \rightharpoonup \Psi^0\mbox{ in } L^2(0,T,H^1(\Omega)) \mbox{ weakly}.\\
\end{gathered}
$$
The same kind of result is established for the drift-diffusion system with homogeneous Neumann boundary conditions by I. Gasser in \cite{Gasser} for a zero doping profile  and  by I. Gasser, C.D. Levermore, P. Markowich, C. Schmeiser in \cite{GLMS} for a regular doping profile. In all these papers, the rigorous proof of the quasi-neutral limit is based on an entropy method. 

The entropy method, described for instance in the review paper \cite{arnold_etal_2004}, has been developed in the last twenty years. It is firstly devoted to the study of the long time behavior of some partial differential equations or systems of partial differential equations and to the study of their equilibrium state.  It consists in looking for a nonnegative Lyapunov functional, called entropy, and its nonnegative dissipation, connected within an entropy-entropy production estimate. Generally, it provides the convergence in relative entropy of the evolutive solution towards an equilibrium state. This method has been widely applied to many different systems: see \cite{arnold_etal_2004} and the references therein, but also  \cite{jungel_M3AS_1995, gajewski_gartner_1996, desvillettes_fellner_2006, glitzky_hunlich_2008, glitzky_gartner_2009}...

However, the entropy method also permits to get new a priori estimates on systems of partial differential equations {\em via} a bound on the entropy production, see \cite{JungelPeng2001, Gasser, GLMS} for instance. In the case of Dirichlet-Neumann boundary conditions, the entropy functional, which has the physical meaning of a free energy, is defined (see \cite{JungelPeng2001}) by

$$\displaylines{
{\mathbb E}(t)=\ds\int_{\Omega} \biggl(H(N)-H(N^D)-\log(N^D)(N-N^D)\hfill\cr
\hfill\vspace{-1.9cm}+H(P)-H(P^D)-\log(P^D)(P-P^D)+ \ds\frac{\lambda^2}{2}\vert \nabla \Psi-\nabla\Psi^D\vert^2 \biggl)dx,}
$$
with $H(x)=\ds\int_1^x \log (t)\  dt= x\log x -x +1$, and the entropy production functional is defined by
$$
{\mathbb I}(t)= \int_{\Omega} \left(N\left\vert\nabla (\log N -\Psi)\right\vert^2+P\left\vert\nabla (\log P +\Psi)\right\vert^2\right) dx dt.
$$
The entropy-entropy production inequality writes:
\begin{equation}\label{ineg-entr-cont}
\ds\frac{d\mathbb E}{dt}(t)+ \frac{1}{2}{\mathbb I}(t)\leq K_D\quad \forall t\geq 0,
\end{equation}
where $K_D$ is a constant depending only on data.
This inequality is crucial in order to perform rigorously the quasi-neutral limit. Indeed, if $ {\mathbb E}(0)$ is uniformly bounded in $\lambda$, \eqref{ineg-entr-cont}  provides a uniform bound on $\int_0^T {\mathbb I}(s) ds$. It implies a priori uniform bounds on $(N^{\lambda}, P^{\lambda}, \Psi^{\lambda})$ solution to $\Pl$ and therefore compactness of a sequence of solutions.

\subsection{Presentation of the numerical method}\label{sec.presscheme}

In order to introduce the numerical scheme for the drift-diffusion system $\Pl$, first, we define the mesh of the  domain $\Omega$. Here, we consider the two-dimensional case but generalization to higher dimensions is straightforward. The mesh  $\cal M=(\T,\E,\cal P)$ is given by $\mathcal T$, a family of open polygonal control volumes, $\mathcal E$, a family of edges and ${\mathcal P}=(x_K)_{K\in\mathcal T}$ a family of points. As it is classical in the finite volume discretization of elliptic or parabolic equations with a two-points flux approximations, we assume that the mesh is admissible in the sense of \cite{EGHbook} (Definition 9.1). It implies that the straight line between two neighboring centers of cell $(x_K,x_L)$ is orthogonal to the edge $\sigma=K|L$ (and therefore collinear to $\nKsig$, the unit normal to $\sigma$ outward to $K$). 

We distinguish in $\E$ the interior edges, $\sigma =K|L$, from the exterior edges, $\sigma\subset \partial\Omega$. Therefore $\E$ is  split into $\E=\E_{int}\cup {\mathcal E}_{ext}$. Within the exterior edges, we distinguish the edges included in $\Gamma^D$ from the edges included in $\Gamma^N$: ${\mathcal E}_{ext}={\mathcal E}_{ext}^D\cup {\mathcal E}_{ext}^N$. For a given control volume $K\in{\mathcal T}$, we define ${\mathcal E}_K$ the set of its edges, which is also split into ${\mathcal E}_K={\mathcal E}_{K,int}\cup{\mathcal E}_{K,ext}^D\cup{\mathcal E}_{K,ext}^N$. For each edge $\sigma\in\E$, there exists at least one cell $K\in\T$ such that $\sigma\in\E_K$. Then, we can denote this cell $K_{\sigma}$. In the case where $\sigma$ is an interior edge ($\sigma=K|L$), $K_{\sigma}$ can be either equal to $K$ or to $L$. 

For all edges $\sigma\in{\mathcal E}$, we define ${\rm d}_{\sigma}={\rm d}(x_K,x_L)$ if $\sigma=K|L\in{\mathcal E}_{int}$ and ${\rm d}_{\sigma}={\rm d}(x_K,\sigma)$ if $\sigma\in{\mathcal E}_{ext}$ with $\sigma\in \E_K$. Then, the transmissibility coefficient is defined by $\tau_{\sigma}={\rm m}(\sigma)/{\rm d}_{\sigma}$, for all $\sigma\in{\mathcal E}$.
We assume that the mesh satisfies the following regularity constraint:
\begin{equation}\label{reg-mesh} 
\exists \xi >0 \mbox{ such that } {\rm d}(x_K,\sigma)\geq \xi \, \mbox{diam}(K),\quad \forall K\in\T, \forall \sigma\in\E_K.
\end{equation}
Let $\beta>0$ be such that $\mbox{card}(\E_K)\leq \beta$ for all $K\in\T$.
Let $\Delta t>0$ be the time step. We set $N_T=E(T/\Delta t)$ and $t^n=n\Delta t$ for all $0\leq n\leq N_T$. The size of the mesh is defined by $\mbox{size }(\T)=\max_{K\in\T}\mbox{diam }(K)$ with
$\mbox{diam}(K)=\sup_{x,y\in K}|x-y|$, for all $K\in\T$. We denote by $\delta=\max(\Delta t, \mbox{size }(\T))$ the size of the space-time discretization. Per se, a finite volume scheme for a conservation law with unknown $u$ provides a vector $u_\T=(u_K)_{K\in\T}\in\R^{\theta}$ (with $\theta={\rm Card}(\T)$)  of approximate values and the associate piecewise constant function, still denoted $u_\T$: 
\begin{equation}\label{defut}
u_\T=\ds\sum_{K\in\T} u_K {\mathbf 1}_K, 
\end{equation}
where ${\mathbf 1}_K$ denotes the characteristic function of the cell $K$. However, since there are  Dirichlet boundary conditions on a part of the boundary, we need to define approximate values for $u$ at the corresponding boundary edges: $u_{\E^D}=(u_{\sigma})_{\sigma\in\E_{ext}^D}\in\R^{\theta^D}$ (with $\theta^D={\rm Card}(\E_{ext}^D)$). Therefore, the vector containing the approximate values in the control volumes and the approximate values at the boundary edges is denoted by  $u_{\cal M}= (u_\T,u_{\E^D})$.

For any vector $u_{\cal M}= (u_\T,u_{\E^D})$, we define, for all $K\in\T$, for all $\sigma\in\E_K$,
\begin{subequations}\label{valeurs_edge}
\begin{align}
&u_{K,\sigma}=\left\{\begin{array}{ll}
u_L,& \mbox{ if } \sigma=K|L\in {\mathcal E}_{K,int},\\
u_{\sigma},& \mbox{ if } \sigma\in {\mathcal E}_{K,ext}^D,\\
u_K,& \mbox{ if } \sigma\in {\mathcal E}_{K,ext}^N,
\end{array}\right.\label{def_uks}\\
&Du_{K,\sigma}=u_{K,\sigma}-u_K \quad \mbox{ and } \quad D_{\sigma}u=\left| Du_{K,\sigma}\right|\label{def_duks}.
\end{align}
\end{subequations}
We also define the discrete $H^1$- semi-norm $|\cdot |_{1,\cal M}$ on the set of approximations  by 
$$
\left| u_{\cal M}\right|^{2}_{1,\cal M}=\sum_{\sigma \in \mathcal{E}}\tau_{\sigma}\,\left(D_{\sigma}u\right)^{2},\quad \forall u_{\cal M}=(u_\T,u_{\E^D}).
$$
As we deal in this paper with a space-time system of equations $\Pl$, we define at each time step, $0\leq n\leq N_T$, the approximate solution $u_\T^n= (u_K^n)_{K\in\T}$ for $u=N,P,\Psi$ and the approximate values at the boundary $u_{\E^D}^n=(u_\sigma^n)_{\sigma\in\E_{ext}^D}$ (which in fact does not depend on $n$ since the boundary data do not depend on time). 
Now, let us present the scheme that will be studied in the sequel. First, we discretize the initial and the boundary conditions. We set 
\begin{equation}\label{CInum}
\Bigl(N_K^0,P_K^0\Bigl)=\ds\frac{1}{{\rm m}(K)}\int_K \Bigl(N_0(x),P_0(x)\Bigl)\, dx, \quad \forall K\in\T,
\end{equation}
$$
\Bigl(N_{\sigma}^D,P_{\sigma}^D,\Psi_{\sigma}^D\Bigl)= \ds\frac{1}{{\rm m}(\sigma)}\int_\sigma \Bigl(N^D(\gamma),P^D(\gamma),\Psi^D(\gamma)\Bigl) d\gamma,\quad
 \quad  \forall \sigma \in {\mathcal E}_{ext}^D.
$$
and we define
\begin{equation}\label{CL_discretes}
N_{\sigma}^n= N_{\sigma}^D,\quad 
P_{\sigma}^n= P_{\sigma}^D,\quad
\Psi_{\sigma}^n= \Psi_{\sigma}^D, \quad  \forall \sigma \in {\mathcal E}_{ext}^D, \forall n\geq 0.
\end{equation}
This means that $N^n_{\E^D}=N^D_{\E^D}$ for all $n\geq 0$.

We consider a Euler implicit in time and finite volume in space discretization.  The scheme  writes:
\begin{subequations}\label{scheme}
\begin{align}
&\m (K)\ds\frac{N_K^{n+1}-N_K^n}{\Delta t}+\ds\sum_{\sigma\in {\mathcal E}_K}{\mathcal F}_{K,\sigma}^{n+1}=0,\quad \forall K\in{\mathcal T}, \forall n\geq 0,\label{scheme-N}\\
&\m (K)\ds\frac{P_K^{n+1}-P_K^n}{\Delta t}+\ds\sum_{\sigma\in {\mathcal E}_K}{\mathcal G}_{K,\sigma}^{n+1}=0,\quad \forall K\in{\mathcal T}, \forall n\geq 0,\label{scheme-P}\\
&-\lambda^2\ds\sum_{\sigma\in {\mathcal E}_K}\tau_{\sigma} D\Psi_{K,\sigma}^{n}=\m (K) (P_K^{n}-N_K^{n}),\quad \forall K\in{\mathcal T}, \forall n\geq 0.\label{scheme-Psi}
\end{align}
\end{subequations}
It remains to define the numerical fluxes  ${\mathcal F}_{K,\sigma}^{n+1}$ and ${\mathcal G}_{K,\sigma}^{n+1}$ which can be seen  respectively as  numerical approximations of $\ds\int_{\sigma} (-\nabla N+N\nabla\Psi)\cdot \nKsig$ and $\ds\int_{\sigma} (-\nabla P-P\nabla\Psi)\cdot \nKsig$ on the interval $[t^n,t^{n+1})$.  We choose to discretize simultaneously the diffusive part and the convective part of the fluxes, by using the Scharfetter-Gummel fluxes. For all $K\in\T$, for all $\sigma\in\E_K$, we set: 
\begin{subequations}\label{FLUX-SG}
\begin{align}
&\label{FLUX-SG-N}
{\mathcal F}_{K,\sigma}^{n+1}=
\tau_{\sigma}\left( B(-D\Psi_{K,\sigma}^{n+1})N_K^{n+1}-B(D\Psi_{K,\sigma}^{n+1})N_{K,\sigma}^{n+1} \right),\\
&\label{FLUX-SG-P}
{\mathcal G}_{K,\sigma}^{n+1}=
\tau_{\sigma}\left( B(D\Psi_{K,\sigma}^{n+1})P_K^{n+1}-B(-D\Psi_{K,\sigma}^{n+1})P_{K,\sigma}^{n+1} \right),
\end{align}
\end{subequations}
where $B$ is the Bernoulli function defined by 
\begin{equation}\label{Bern}
B(0)=1 \mbox{ and } B(x)=\ds\frac{x}{\exp (x) -1}\  \forall x\neq 0.
\end{equation}

These fluxes have been introduced by A. M. Il'in in \cite{ilin} and D. L. Scharfetter and H.~K.~Gummel in \cite{SG} for the numerical approximation of convection-diffusion terms with linear diffusion.  It has been established by R.~Lazarov, I.~Mishev and P.~Vassilevsky in \cite{LMV} that they are second-order accurate in space. Moreover, they preserve steady-states. Dissipativity of the Scharfetter-Gummel scheme with a backward Euler time discretization was proved in \cite{gajewski_gartner_1996}. A  proof of the exponential decay of the free energy along trajectories towards the thermodynamic equilibrium on boundary conforming Delaunay grids was also given by A.~Glitzky in \cite{glitzky08,glitzky11}. In \cite{gartner09}, K.~G\"artner establishes some bounds for discrete steady states solutions obtained with the Scharfetter-Gummel scheme. Moreover, M.~Chatard proved in \cite{Marianne-FVCA6} a discrete entropy estimate, with control of the entropy production, which yields the long-time behavior of the Scharfetter-Gummel scheme for the drift-diffusion system. The generalization of the Scharfetter-Gummel fluxes to nonlinear diffusion has been studied by A.~J\"ungel and P.~Pietra in \cite{JungelPietra}, R.~Eymard, J.~Fuhrmann and K.~G\"{a}rtner in \cite{EFG} and M.~Bessemoulin-Chatard in \cite{MarianneSG}.

\begin{remark}
Let us note that the definition \eqref{valeurs_edge} ensures that $D\Psi_{K,\sigma}^{n+1}= 0$ and also that  ${\mathcal F}_{K,\sigma}^{n+1}={\mathcal G}_{K,\sigma}^{n+1}=0$,   for all $\sigma\in\E_{K,ext}^N$. These relations are  consistent with the Neumann boundary conditions \eqref{CL_neum}. 
\end{remark}
\smallskip

In the sequel, we denote by $\Sl$ the scheme \eqref{CInum}--\eqref{Bern}. It is a fully implicit in time scheme:  the numerical solution $(\NKnp,\PKnp,\Psi_K^{n+1})_{K\in\T}$ at each time step is defined as  a solution of the nonlinear system of equations \eqref{scheme}--\eqref{FLUX-SG}. When choosing $D\Psi_{K,\sigma}^n$ instead of $D\Psi_{K,\sigma}^{n+1}$ in the definition of the fluxes \eqref{FLUX-SG}, we would get a decoupled scheme whose solution is obtained by solving successively three linear systems of equations for $N$, $P$ and $\Psi$. However, this other choice of time discretization used in \cite{CCHLP,CCHP} induces a stability condition of the form
$\Delta t\leq C \lambda^2$ (see for instance \cite{corrosion}). Therefore, it cannot be used in practice for small values of $\lambda$ and it does not preserve the quasi-neutral limit.

Setting $\lambda=0$ in the scheme $\Sl$ leads to the scheme $\Sz$ defined in the following. The scheme for the Poisson equation \eqref{scheme-Psi} becomes $P_{K}^n-N_{K}^n=0$ for all $K\in\T$, $n\in\N$. In order to avoid any incompatibility condition at $n=0$ (which would correspond to an initial layer), we assume that the initial  conditions $N_0$ and $P_0$ satisfy the quasi-neutrality assumption:
\begin{equation}\label{QN-CI}
P_0-N_0=0.
\end{equation}
Adding and subtracting \eqref{scheme-N} and \eqref{scheme-P}, and using $P_{K}^n=N_{K}^n$ for all $K\in\T$ and  $n\in\N$, we get   
$$
\begin{aligned}
\m(K) \ds\frac{N_K^{n+1}-N_K^n}{\Delta t}+\frac{1}{2}\ds\sum_{\sigma\in {\mathcal E}_K}\left(\FKsig^{n+1}+\GKsig^{n+1}\right)=0, \forall K\in\T, \forall n\geq 0,\\
\mbox{ and }\sum_{\sigma\in\E_K} \left(\FKsig^{n+1}-\GKsig^{n+1}\right)=0, \forall K\in\T, \forall n\geq 0.
\end{aligned}
$$
But, using the following property of the Bernoulli function 
\begin{equation}\label{propBern}
B(x)-B(-x)=-x \quad \forall x\in\R,
\end{equation}
we have,  $\forall K\in\T, \forall \sigma\in\E_{K,int}\cup \E_{K,ext}^N$:
$$
\begin{aligned}
&\FKsig^{n+1}-\GKsig^{n+1}=\ts D\Psi_{K,\sigma}^{n+1}(N_K^{n+1}+N_{K,\sigma}^{n+1}), \\
\mbox{ and } \quad&\FKsig^{n+1}+\GKsig^{n+1}=-\ts \left(B(D\Psi_{K,\sigma}^{n+1})+B(-D\Psi_{K,\sigma}^{n+1})\right)DN_{K,\sigma}^{n+1}.
\end{aligned}
$$
Let us note that these equalities still hold for each Dirichlet boundary edge $\sigma\in\E_{K,ext}^D$ if $N_\sigma^D=P_\sigma^D$. In the sequel, when studying the scheme at the quasi-neutral limit $\Sz$, we assume  the quasi-neutrality of the initial conditions \eqref{QN-CI} and of the boundary conditions:
\begin{equation}\label{QN-CL}
P^D-N^D=0.
\end{equation}
Finally, the scheme $\Sz$ can be rewritten: $\forall K\in\T$, $\forall n\geq 0$,
\begin{subequations}\label{Sz}
\begin{gather}
{\m} (K)\ds\frac{N_K^{n+1}-N_K^n}{\Delta t}-\ds\sum_{\sigma\in {\mathcal E}_K}\ts \frac{B(D\Psi_{K,\sigma}^{n+1})+B(-D\Psi_{K,\sigma}^{n+1})}{2}DN_{K,\sigma}^{n+1}=0,\label{Sz-N}\\
-\sum_{\sigma\in\E_K}\ts D\Psi_{K,\sigma}^{n+1}(N_K^{n+1}+N_{K,\sigma}^{n+1})=0,\label{Sz-Psi}\\
P_{K}^n-N_{K}^n=0,\label{Sz-P}
\end{gather}
\end{subequations}
with the initial conditions \eqref{CInum} and the boundary conditions \eqref{CL_discretes}.

\subsection{Main results and outline of the paper}

The scheme $\Sl$ is implicit in time. Then we begin by proving that the nonlinear system of equations \eqref{scheme}  admits a solution at each time step.  The proof of this result is based on the application of Brouwer's fixed point theorem. 
The existence result is given in Theorem \ref{thm-ex}
and is proved in Section \ref{sec.ex}.

\begin{theorem}[Existence of a solution to the numerical scheme]\label{thm-ex}
~

We assume Hypotheses \ref{HYP}, let $\T$ be an admissible mesh of $\Omega$ satisfying \eqref{reg-mesh}  and $\Delta t>0$. If $\lambda=0$, we further assume the quasi-neutrality of the initial and boundary conditions \eqref{QN-CI} and \eqref{QN-CL}. Then, for all $\lambda\geq 0$, there exists a solution to the scheme $\Sl$: $(N_K^n, P_K^n, \Psi_K^n)_{K\in\T}\in(\R^{\theta})^3$ for all $n\geq 0$.  Moreover, the approximate densities satisfy the following $L^{\infty}$ estimate:
\begin{equation}\label{estLinf}
\forall K\in\T, \forall n\geq 0,\   m\leq N_K^n, P_K^n\leq M.
\end{equation}
\end{theorem}
Then, in Section \ref{sec.entropy}, we prove the discrete counterpart of the entropy-dissipation inequality \eqref{ineg-entr-cont}.  As the functions $N^D$, $P^D$, $\Psi^D$ are given on the whole domain, we can set:
$$
\Bigl(N_{K}^D,P_{K}^D,\Psi_{K}^D\Bigl)= \ds\frac{1}{{\rm m}(K)}\int_K \Bigl(N^D(x),P^D(x),\Psi^D(x)\Bigl) dx,\quad
 \ \   \forall K \in \T.
$$
For all $n\in{\mathbb N}$, the discrete  entropy functional is  defined by: 
\begin{multline*}
\mathbb E^{n} = \sum_{K \in \mathcal{T}}\text{m}(K) \left(H(N^{n}_{K})-H(N^{D}_{K})-\log(N^{D}_{K})\left(N^{n}_{K}-N^{D}_{K}\right)\right)\\
 + \sum_{K \in \mathcal{T}}\text{m}(K) \left(H(P^{n}_{K})-H(P^{D}_{K})-\log(P^{D}_{K})(P^{n}_{K}-P^{D}_{K})\right)+ \frac{\lambda^{2}}{2}\left|\Psi_{\cal M}^{n}-\Psi_{\cal M}^{D}\right|^{2}_{1,{\cal M}},
\end{multline*}
and the discrete entropy production is defined by 
$$\begin{array}{l}
\ds\mathbb I^{n} = \somsig\tau_{\sigma}\Biggl[
\min\left(N_{K}^{n},N_{K,\sigma}^{n}\right)\Bigl(D_\sigma\left(\log N^n-\Psi^n\right)\Bigl)^{2}\hspace{3cm}~\\
[-15pt]\ds\hspace{5cm}+\min\left(P_{K}^{n},P_{K,\sigma}^{n}\right)\Bigl(D_\sigma\left(\log P^n+\Psi^n \right)\Bigl)^{2}\Biggl],
\end{array}
$$
where the notation $\somsig$ means a sum over all the edges $\sigma\in\E$ and $K$ inside the sum is replaced by $K_{\sigma}$ (therefore, $\sigma$ is an edge of the cell $K=K_{\sigma}$). 

The discrete counterpart of \eqref{ineg-entr-cont} is given in Theorem \ref{thm-entropy}
\begin{theorem}[Discrete entropy-dissipation inequality]\label{thm-entropy}
~

We assume Hypotheses \ref{HYP}, let $\T$ be an admissible mesh of $\Omega$ satisfying \eqref{reg-mesh}  and $\Delta t>0$. Then, there exists $K_E$, depen\-ding only on $\Omega$, $T$, $m$, $M$, $N^D$, $P^D$, $\Psi^D$, $\beta$  and $\xi$  such that, for all  $\lambda\geq 0$, a solution to the scheme $\Sl$,  $(N_\T^{n},P_\T^n,\Psi_\T^n)_{0\leq n\leq N_T}$, satisfies the following inequality: 
\begin{subequations}\label{entr-dissip}
\begin{equation}\label{entropy-estimate}
\ds\frac{\mathbb E ^{n+1}-\mathbb E ^{n}}{\Delta t}+\frac{1}{2}\, \mathbb I^{n+1} \leq K_E,\quad \forall n\geq 0.
\end{equation}
Furthermore, if $N^0$ and $P^0$ satisfy the quasi-neutrality assumption \eqref{QN-CI}, we have 
\begin{equation}\label{diss-entropy-estimate}
 \sum_{n=0}^{N_T-1} \Delta t \,  \mathbb I^{n+1}\leq K_E(1+\lambda^2).
\end{equation}
 \end{subequations}
\end{theorem}
Let us note that the last inequality \eqref{diss-entropy-estimate}, which ensures the control of the discrete entropy production, 
depends on $\lambda$. However, as we are interested in the quasi-neutral limit $\lambda\to 0$, we can assume that $\lambda$ stays in a bounded interval $[0,\lambda_{max}]$ and then get a uniform bound in $\lambda$.

In Section \ref{sec.estimates}, we show how to obtain, from the discrete entropy-dissipation inequality, all the a priori estimates needed  for the convergence of the scheme. These estimates are given in the following Theorem \ref{thm-estap}. There are weak-BV inequality \eqref{estBV} and $L^2(0,T,H^1)$-estimates\eqref{estL2H1_NP} on $N$ and $P$ and $L^2(0,T,H^1)$-estimates \eqref{estL2H1_Psi} on $\Psi$. 

\begin{theorem}[A priori estimates satisfied by the approximate solution]\label{thm-estap}
We assume Hypotheses \ref{HYP}, let $\T$ be an admissible mesh of $\Omega$ satisfying \eqref{reg-mesh}  and $\Delta t>0$. 
We also assume that the initial and boundary conditions satisfy the quasi-neutrality relations \eqref{QN-CI} and \eqref{QN-CL}. Then, there exists a constant $K_F$ depending only on $\Omega$, $T$, $m$, $M$, $N^D$, $P^D$, $\Psi^D$, $\beta$ and $\xi$ such that, for all $\lambda\geq 0$, a solution to the scheme $\Sl$,  $(N_\T^{n},P_\T^n,\Psi_\T^n)_{0\leq n\leq N_T}$, satisfies the following inequalities: 
\begin{subequations}
\begin{gather}
\somn\Delta t\sum_{\sigma\in\E} \ts D_\sigma\Psi^{n+1}\biggl((D_\sigma P^{n+1})^2+(D_\sigma N^{n+1})^2\biggl) \leq K_{F}(1+\lambda^2),\label{estBV}\\ 
\somn \Delta t\sum_{\sigma\in\E}\ts (D_\sigma N^{n+1})^2+
\somn \Delta t\sum_{\sigma\in\E}\ts (D_\sigma P^{n+1})^2\leq K_F(1+\lambda^2),\label{estL2H1_NP}\\
\somn \Delta t\sum_{\sigma\in\E} \ts (D_\sigma\Psi^{n+1})^2\leq K_F(1+\lambda^2).\label{estL2H1_Psi}
\end{gather}
\end{subequations}
\end{theorem}
Estimates \eqref{estL2H1_NP} and \eqref{estL2H1_Psi} yield the compactness of a sequence of approximate solutions, as shown for instance in \cite{CCHLP}, applying some arguments developed in \cite{EGHbook}. To prove the convergence of the numerical method, it remains to pass to the limit in the scheme and by this way prove that the limit of the sequence of approximate solutions is a weak solution to $\Sl$. It can still be done as in \cite{CCHLP}, but dealing with the Scharfetter-Gummel fluxes as in \cite{MarianneSG}. The convergence proof is not detailed in this paper. Let us just note that the convergence proof holds for all $\lambda\geq 0$. 

Finally, in Section \ref{sec.numexp}, we present some numerical experiments. They illustrate the stability of  the scheme when $\lambda$ varies and goes to $0$. They show that the proposed scheme is an asymptotic-preserving scheme in the quasi-neutral limit since the scheme order in space and time is preserved uniformly in the limit. Moreover, let us emphasize that although our results are proved under the restrictive assumption of vanishing doping profile, the numerical results show that the error estimates remain independent of $\lambda$, even for piecewise constant doping profiles.

\section{Existence of a solution to the numerical scheme}\label{sec.ex}

In this Section, we prove Theorem \ref{thm-ex} (existence of a solution to the numerical scheme $\Sl$ for all $\lambda\geq 0$). 
As the boundary conditions are explicitly defined by \eqref{CL_discretes},
it consists in proving at each time step the existence of $(N_\T^n,P_\T^n,\Psi_\T^n)$ solution to the nonlinear system of equations \eqref{scheme} when $\lambda>0$ or \eqref{Sz} when $\lambda=0$. We distinguish the two cases in the proof.

\subsection{Study of the case $\lambda>0$}

We consider here $\lambda>0$. The proof of Theorem \ref{thm-ex} is done by induction on $n$. 
The vectors $N_\T^0$ and $P_\T^0$ are defined by \eqref{CInum} and $\Psi_\T^0$ by \eqref{scheme-Psi}. Furthermore, the hypothesis on the initial data \eqref{HYP-mM} ensures that 
$$
m\leq N_K^0, P_K^0\leq M \quad \forall K\in\T.
$$
We suppose that, for some $n\geq 0$, $(N_\T^n,P_\T^n,\Psi_\T^n)$ is known and satisfies the $L^{\infty}$ estimate \eqref{estLinf}. We want to establish  the existence of $(N_\T^{n+1},P_\T^{n+1},\Psi_\T^{n+1})$ solution to the nonlinear system of equations \eqref{scheme}, also satisfying \eqref{estLinf}.  Therefore, we follow some ideas developed by A. Prohl and M. Schmuck in \cite{ProhlSchmuck2009} and used by C. Bataillon {\em et al} in \cite{corrosion}. This method consists in introducing a problem penalized by an arbitrary parameter which will be conveniently chosen.

Let $ \mu>0$, we introduce an application $
{ T}_{\mu}^{n}: \R^{\theta}\times \R^{\theta}   \rightarrow  \R^{\theta} \times  \R^{\theta}$, such that 
${ T}_{\mu}^{n}(N_\T  , P_\T) = (\widehat{N}_\T , \widehat{P}_\T)$,
based on a linearization of the scheme \eqref{scheme} and defined in two steps. 
\begin{itemize}
\item First, we define $\Psi_\T \in \R^{\theta}$ as the solution to the following linear system: 
\end{itemize}
\begin{subequations}\label{Tmu-schemepsi}
\begin{gather}
-\lambda^{2}\sum_{\sigma \in \E_{K}} \ts \DPsiKs=\text{m}(K)(P_{K}-N_{K}),\quad \forall K\in\T,\\
\mbox{ with } \Psi_{\sigma}=\Psi_{\sigma}^{D},\quad\forall \sigma\in \E_{ext}^{D}.
\end{gather}
\end{subequations}
\begin{itemize}
\item Then, we construct $(\widehat{N}_\T,\widehat{P}_\T)$ as the solution to the following linear scheme:
\end{itemize}
\begin{subequations}\label{Tmu-schemetilde}
\begin{align}
&\ds\frac{\m (K)}{\Delta t} \left(\left(1+\frac{\mu}{\lambda^2}\right)\,\widehat N_K-\frac{\mu}{\lambda^2}\,N_K-N_K^n\right)\nonumber\\
&\hspace{1.5cm}+\ds\sum_{\sigma\in {\mathcal E}_K}\ts \left( B\left(-\DPsiKs\right)\widehat{N}_{K}-B\left(\DPsiKs\right)\widehat{N}_{K,\sigma}\right) =0,\quad \forall K\in{\mathcal T},\label{scheme-Ntilde}\\
&\ds\frac{\m (K)}{\Delta t} \left(\left(1+\frac{\mu}{\lambda^2}\right)\,\widehat P_K-\frac{\mu}{\lambda^2}\,P_K-P_K^n\right)\nonumber\\
&\hspace{1.5cm}+\ds\sum_{\sigma\in {\mathcal E}_K}\ts \left( B\left(\DPsiKs\right)\widehat{P}_{K}-B\left(-\DPsiKs\right)\widehat{P}_{K,\sigma}\right) =0,\quad \forall K\in{\mathcal T},\label{scheme-Ptilde}\\
&\mbox{ with } \widehat{N}_{\sigma}=N_{\sigma}^{D}\  \mbox{�and }� \widehat{P}_{\sigma}=P_{\sigma}^{D} \quad \forall 
\sigma \in \E_{ext}^{D}.\label{schemetilde_bc}
\end{align}
\end{subequations}
 The existence and uniqueness of $\Psi_\T$ solution to the linear system \eqref{Tmu-schemepsi} are obvious. The second step \eqref{Tmu-schemetilde} also leads to two decoupled linear systems which can be written under a matricial form: ${\mathbb A}_{N}\, \widehat{N}_\T={\mathbb S}_{N}^n$ and ${\mathbb A}_{P}\widehat{P}_\T={\mathbb S}_P^n$ . The matrix ${\mathbb A}_{N}$ for instance is the sparse matrix defined by  
\begin{align*}
({\mathbb A}_{N})_{K,K} &= \frac{\text{m}(K)}{\Delta t}\left(1+\frac{\mu}{\lambda^{2}}\right)+\sum_{\sigma \in \E_{K}\setminus \E_{K,ext}^N}\ts\,B\left(-\DPsiKs\right) \quad \forall K \in \T, \\
({\mathbb A}_{N})_{K,L} &= -\ts \, B \left(\DPsiKs \right) \quad \forall L \in \T \text{ such that } \sigma=K|L \in \E_{int}.
\end{align*} 
We verify that ${\mathbb A}_{N}$ has positive diagonal terms, nonpositive offdiagonal terms and is strictly diagonally dominant with respect to its columns. It implies that ${\mathbb A}_{N}$ is an M-matrix: it is invertible and its inverse has only nonnegative coefficients. The same result holds for ${\mathbb A}_{P}$. Thus, we obtain that the scheme \eqref{Tmu-schemetilde} admits a unique solution $(\widehat{N}_\T,\widehat{P}_\T)\in \R^{\theta}\times \R^{\theta}$, so that the application $T_{\mu}^n$ is well defined and is a continuous application.

Now, in order to apply Brouwer's fixed point theorem, we want to prove that $T_{\mu}^{n}$ preserves the set
\begin{equation}\label{defiC}
\mathcal{C}_{m,M}=\left\{(N_\T,P_\T) \in \R^{\theta}\times\R^{\theta}; \quad m \leq N_{K}, P_{K}\leq M, \quad \forall K \in \T \right\}.
\end{equation}
The right hand side of the linear system \eqref{scheme-Ntilde} is defined by 
 \begin{equation*}
({\mathbb S}_{N})_{K}=\frac{\text{m}(K)}{\Delta t}\left(\NKn+\frac{\mu}{\lambda^{2}}\,N_{K}\right)+\ds{\sum_{\sigma \in \E_{K,ext}^{D}}}\ts\,B\left(\DPsiKs\right)\,{N}_{\sigma}^D, \quad  \forall K \in \T.
\end{equation*}
If $N_\T \geq 0$, we have ${\mathbb S}_N^n\geq 0$ and, as ${\mathbb A}_N$ is an M-matrix, we get  $\widehat N_\T\geq 0$. 
Similarly, if $P_\T \geq 0$, we obtain that $\widehat P_\T\geq 0$.

In order to prove that $\widehat N_K\leq M$ for all $K\in\T$, we introduce ${\mathbf M}_\T$ the constant vector of $R^{\theta}$  with unique value $M$ and we compute ${\mathbb A}_{N}(\widehat{N}_\T-\mathbf{M}_\T)$. Using the property \eqref{propBern}, we get that for all $K \in \T$,
\begin{multline*}
\left({\mathbb A}_{N}(\widehat{N}_\T-\mathbf{M}_\T)\right)_{K}= \,\,\,\frac{\text{m}(K)}{\Delta t}(\NKn-M)+\frac{\text{m}}{\Delta t}\,\frac{\mu}{\lambda^{2}}(N_{K}-M) \\
+\sum_{\sigma \in \E_{K,ext}^{D}}\ts \left(B \left(\DPsiKs\right)N_{\sigma}^{D}-B \left(-\DPsiKs\right)M \right) 
-\,\,\, M \sum_{\sigma \in \E_{K,int}}\ts \, \DPsiKs.
\end{multline*}
Since $B$ is a nonnegative function and $N^D$ satisfies \eqref{HYP-mM}, we have, for all $\sigma \in \E_{K,ext}^{D}$, 
$$
B \left(\DPsiKs\right)N_{\sigma}^{D}-B \left(-\DPsiKs\right)M= \, B \left(\DPsiKs\right)(N_{\sigma}^{D}-M)-\DPsiKs \, M 
\leq  \, -\DPsiKs \, M $$
Then, using the induction assumption $\NKn \leq M$ for all $K\in\T$, it yields
\begin{equation*}
\left({\mathbb A}_{N}(\widehat{N}_\T-\mathbf{M}_\T)\right)_{K} \leq \frac{\text{m}(K)}{\Delta t}\,\frac{\mu}{\lambda^{2}}\,(N_{K}-M)-M \sum_{\sigma \in \E_{K}}\ts \, \DPsiKs,
\end{equation*}
and using \eqref{Tmu-schemepsi}, we get that for all $K\in\T$
\begin{equation}\label{majNlin}
\left({\mathbb A}_{N}(\widehat{N}_\T-\mathbf{M}_\T)\right)_{K} \leq \frac{\text{m}(K)}{\lambda^{2}}\left(\frac{\mu}{\Delta t}-M\right)(N_{K}-M)+M\,\frac{\text{m}(K)}{\lambda^{2}}(P_{K}-M).
\end{equation}
We can prove exactly in the same way that, for all $K\in\T$,
\begin{equation}\label{minNlin}
\left({\mathbb A}_{N}(\widehat{N}_\T-\mathbf{m}_\T)\right)_{K} \geq \frac{\text{m}(K)}{\lambda^{2}}\left(\frac{\mu}{\Delta t}-m\right)(N_{K}-m)+m\,\frac{\text{m}(K)}{\lambda^{2}}(P_{K}-m) .
\end{equation}
Now, let us choose $\mu$ in order to obtain the expected $L^{\infty}$ properties. Indeed, since $\mu>0$ is an arbitrary constant, we can choose it such that $m \, \Delta t \leq M \, \Delta t \leq \mu$ without any constraint on the time step. Then, if $(N_\T,P_\T)\in{\mathcal C}_{m,M}$, inequalities \eqref{majNlin} and \eqref{minNlin} imply that 
$$
{\mathbb A}_{N}(\widehat{N}_\T-\mathbf{M}_\T)\leq 0 \mbox{ and } {\mathbb A}_{N}(\widehat{N}_\T-\mathbf{m}_\T)\geq 0. 
$$
As ${\mathbb A}_{N}$ is an M-matrix, we conclude that $m\leq \widehat N_K\leq M$ for all $K\in \T$. The proof that 
$m\leq \widehat P_K\leq M$ for all $K\in \T$ is similar and we have $(\widehat N_\T,\widehat P_\T)\in {\mathcal C}_{m,M}$.

Finally, $T_{\mu}^{n}$ is a continuous application which stabilizes the set $\mathcal{C}_{m,M}$. Then, by the Brouwer's fixed-point theorem, $T_{\mu}^{n}$ has a fixed point  in $\mathcal{C}_{m,M}$ which is denoted by $(N_\T^{n+1}, P_\T^{n+1})$ and satisfies the $L^{\infty}$ estimate \eqref{estLinf}. The corresponding $\Psi_\T$ defined by \eqref{Tmu-schemepsi} is denoted by $\Psi_\T^{n+1}$ and $( N_\T^{n+1}, P_\T^{n+1}, \Psi_\T^{n+1})$ is a solution to the scheme \eqref{scheme}. This shows Theorem \ref{thm-ex} when $\lambda>0$. 

\subsection{Study of the case $\lambda=0$}

Now, we prove Theorem \ref{thm-ex} when $\lambda=0$. In this case, thanks to the quasi-neutrality assumptions, we have shown that the scheme $\Sz$ rewrites as the nonlinear system of equations \eqref{Sz}. Indeed, it is sufficient to study the system \eqref{Sz-N}-\eqref{Sz-Psi}, whose unknowns are $(N_\T^{n+1}, \Psi_\T^{n+1})$. 

The proof is done by induction as in the case $\lambda>0$. Let us first note that $N_\T^0$ satisfy the $L^{\infty}$ estimate \eqref{estLinf}. Then, we assume that, for $n\geq 0$,  $N_\T^n$ is known and also satisfies \eqref{estLinf}. As in the case $\lambda>0$, we introduce an application
${ T}^{n}: (\R_+^{\ast})^{\theta}   \rightarrow  \R^{\theta} $ such that ${ T}^{n}(N_\T)= \widehat{N}_\T$, 
based on a linearization of \eqref{Sz-N}-\eqref{Sz-Psi} and defined in two steps. 
\begin{itemize}
\item First, we define $\Psi_\T\in\R^{\theta}$ as the solution to the linear system:
\end{itemize}
\begin{subequations}\label{T-Psi}
\begin{gather}
-\sum_{\sigma\in\E_K}\ts D\Psi_{K,\sigma}(N_K+N_{K,\sigma})=0,\quad \forall K\in{\mathcal T},\label{T_Psi_int}\\
\mbox{ with } \Psi_\sigma=\Psi_\sigma^D\quad  \forall \sigma\in\E_{ext}^D.
\end{gather}
\end{subequations}
\begin{itemize}
\item Then, we define ${\widehat N}_\T\in\R^{\theta}$ as the solution to the linear system:
\end{itemize}
\begin{subequations}\label{T-N}
\begin{gather}
\frac{\m(K)}{\Delta t}({\widehat N}_K-N_K^n)-\sum_{\sigma\in\E_K} \ts \frac{ B(D\Psi_{K,\sigma})+ B(-D\Psi_{K,\sigma})}{2} D{\widehat N}_{K,\sigma}=0, \forall K\in{\mathcal T},\\
\mbox{ with } {\widehat N}_\sigma=N_\sigma^D \quad\forall \sigma\in\E_{ext}^D.
\end{gather}
\end{subequations}
First, let us prove that the application $T^n$ is well defined. If $N_K>0$ for all $K\in\T$, the matrix of the linear system \eqref{T-Psi} is a positive symmetric definite matrix (it can be proved for instance by multiplying \eqref{T_Psi_int} by $\Psi_K$ and summing over $K\in\T$). Therefore, $\Psi_\T$ is uniquely defined. 

The linear system \eqref{T-N} can be written under the matricial form ${\mathbb A}_N \widehat N_\T=\mathbb S_N^n$ where 
the matrix ${\mathbb A}_N$ is  defined by:
\begin{align*}
({\mathbb A}_{N})_{K,K} &= \frac{\m(K)}{\Delta t} +\frac{1}{2}\sum_{\sigma\in\E_K\setminus\E_{K,ext}^N}\ts\left(B(D\Psi_{K,\sigma})+ B(-D\Psi_{K,\sigma})\right) \quad \forall K \in \T, \\
({\mathbb A}_{N})_{K,L} &=- \frac{\ts}{2} \left(B(D\Psi_{K,\sigma})+ B(-D\Psi_{K,\sigma})\right) \quad \forall L \in \T \text{ such that } \sigma=K|L \in \E_{int}.
\end{align*} 
and the right hand side $\mathbb S_N^n$ is defined by:
 \begin{equation*}
({\mathbb S}_{N}^{n})_{K}= \frac{\m(K)}{\Delta t}N_K^n+\frac{1}{2}\sum_{\sigma\in\E_{K,ext}^D}\ts \left(B(D\Psi_{K,\sigma})+ B(-D\Psi_{K,\sigma})\right)N_\sigma^D\quad  \forall K \in \T.
\end{equation*}
The matrix ${\mathbb A}_N$ is an M-matrix because it has positive diagonal terms, nonpositive off diagonal terms and it is strictly diagonally dominant with respect to its columns. Therefore the linear system \eqref{T-N} has a unique solution $\widehat N_\T$, so that the application  $T^n$ is well defined. It is also continuous. 

Now, let us prove that $T^n$ preserves the set 
$$
{\mathcal K}_{m,M}=\left\{ N_\T\in\R^{\theta}; \ m\leq N_K\leq M, \ \forall K\in\T\right\}.
$$
Therefore, we compute ${\mathbb A}_{N}(\widehat{N}_\T-\mathbf{M}_\T)$. We obtain 
$$
\begin{gathered}
\left({\mathbb A}_{N}(\widehat{N}_\T-\mathbf{M}_\T)\right)_K=\ds\frac{\m(K)}{\Delta t}(N_K^n-M)\hspace{6cm}\\
\hspace{3cm}+\frac{1}{2}\sum_{\sigma\in\E_{K,ext}^D}\ts \left(B(D\Psi_{K,\sigma})+ B(-D\Psi_{K,\sigma})\right)(N_\sigma^D-M).
\end{gathered}
$$
Thanks to the induction hypothesis and \eqref{HYP-mM}, we deduce that ${\mathbb A}_{N}(\widehat{N}_\T-\mathbf{M}_\T)\leq 0$. Similarly,  we prove that ${\mathbb A}_{N}(\widehat{N}_\T-\mathbf{m}_\T)\geq 0$. This implies $\widehat{N}_\T\in {\mathcal K}_{m,M}$. We conclude the proof of Theorem \ref{thm-ex} in the case $\lambda=0$ by applying the Brouwer's fixed point theorem as in the case $\lambda>0$.

\section{Discrete entropy-dissipation inequality}\label{sec.entropy}

In this Section, we prove Theorem \ref{thm-entropy}. Therefore, we adapt the proof done by M. Chatard in \cite{Marianne-FVCA6} for the study of the long-time behavior of the scheme (in this case, the entropy functional is defined relatively to the thermal equilibrium). 

Since $H$ is a convex function, we have $\mathbb E ^n\geq 0$ and $\mathbb E ^{n+1}-\mathbb E ^{n} \leq T_{1}+T_{2}+T_{3}$, with 
\begin{eqnarray*}
T_{1} &=& \sum_{K\in \mathcal{T}}\mK\left(\log\left(N_{K}^{n+1}\right)-\log\left(N_K^D\right)\right)\left(N_{K}^{n+1}-N_{K}^{n}\right),\\
T_{2} &=& \sum_{K\in \mathcal{T}}\mK\left(\log\left(P_{K}^{n+1}\right)-\log\left(P_K^D\right)\right)\left(P_{K}^{n+1}-P_{K}^{n}\right),\\
T_{3} &=& \frac{\lambda^{2}}{2}\left|\Psi^{n+1}_{\cal M}-\Psi_{\cal M}^{D}\right|^{2}_{1,_{\cal M}}-\frac{\lambda^{2}}{2}\left|\Psi_{\cal M}^{n}-\Psi_{\cal M}^{D}\right|^{2}_{1,_{\cal M}}
\end{eqnarray*}

Multiplying the scheme on $N$~\eqref{scheme-N} by $\Delta t\left(\log\left(N_{K}^{n+1}\right)-\log\left(N_K^D\right)\right)$, summing 
over $K\in\T$ and following a discrete integration by parts (using \eqref{CL_discretes}), we rewrite $T_1$:
\begin{equation}\label{t1}
\begin{aligned}
T_1&=-\Delta t \sum_{K\in\T} \sum_{\sigma\in\E_K} \FKsig^{n+1}\left(\log\left(N_{K}^{n+1}\right)-\log\left(N_K^D\right)\right)\\
&=\Delta t \somsig\FKsig^{n+1}\left((D\log N^{n+1})_{K,\sigma}-(D\log N^D)_{K,\sigma}\right).
\end{aligned}
\end{equation}
Starting from the scheme on $P$~\eqref{scheme-N} and following the same kind of computations,  we also rewrite $T_2$:
\begin{equation}\label{t2}
T_2=\Delta t \somsig\GKsig^{n+1}\left((D\log P^{n+1})_{K,\sigma}-(D\log P^D)_{K,\sigma}\right).
\end{equation}

Now, in order to estimate $T_3$, we subtract  two consecutive time steps of the scheme on $\Psi$~\eqref{scheme-Psi}. It yields:
$$
-\lambda^2 \ds\sum_{\sigma\in\E_K} \ts D\Psi_{K,\sigma}^{n+1}+ \lambda^2 \ds\sum_{\sigma\in\E_K} \ts D\Psi_{K,\sigma}^{n}=\mK\Bigl((P_K^{n+1}-P_K^n)-(N_K^{n+1}-N_K^n)\Bigl).
$$
Thanks to the schemes on $N$~\eqref{scheme-N} and $P$~\eqref{scheme-P}, it rewrites
$$
\lambda^2 \ds\sum_{\sigma\in\E_K} \ts (D\Psi_{K,\sigma}^{n+1}-D\Psi_{K,\sigma}^{D})- \lambda^2 \ds\sum_{\sigma\in\E_K} \ts (D\Psi_{K,\sigma}^{n}-D\Psi_{K,\sigma}^{D})=\Delta t\sum_{\sigma\in\E_K}(\GKsig^{n+1}-\FKsig^{n+1}).
$$

Multiplying this equality by $\Psi_K^{n+1}-\Psi_K^D$, summing over  $K\in\T$, integrating by parts and using the boundary conditions, we obtain:
\begin{multline*}
\lambda^2\!\!\somsig \ts (\DPsiKs^{n+1}-\DPsiKs^D)^2-\lambda^2\!\!\somsig \ts (\DPsiKs^{n}-\DPsiKs^D)(\DPsiKs^{n+1}-\DPsiKs^D)=\\
\Delta t\somsig(\GKsig^{n+1}-\FKsig^{n+1})(\DPsiKs^{n+1}-\DPsiKs^D).
\end{multline*}
But, for all $K\in\T$ and all $\sigma\in\E_K$, we have
$$
-(\DPsiKs^{n}-\DPsiKs^D)(\DPsiKs^{n+1}-\DPsiKs^D)\geq\! \ds-\frac{1}{2}  (\DPsiKs^{n}-\DPsiKs^D)^2
-\ds\frac{1}{2} (\DPsiKs^{n+1}-\DPsiKs^D)^2 
$$
and therefore for all $\lambda\geq 0$
\begin{equation}\label{t3}
T_3\leq \Delta t\somsig(\GKsig^{n+1}-\FKsig^{n+1})(\DPsiKs^{n+1}-\DPsiKs^D).
\end{equation}
From \eqref{t1}, \eqref{t2} and \eqref{t3}, we get 
$$\begin{aligned}
\frac{\mathbb E ^{n+1}-\mathbb E ^{n}}{\Delta t}\leq&\somsig\Biggl[\FKsig^{n+1}\left(D(\log N-\Psi)_{K,\sigma}^{n+1}-D(\log N-\Psi)_{K,\sigma}^{D}\right)\\
&+\GKsig^{n+1}\left(D(\log P+\Psi)_{K,\sigma}^{n+1}-D(\log P+\Psi)_{K,\sigma}^{D}\right)\Biggl].\end{aligned}$$
But, thanks to inequalities \eqref{maj1} and \eqref{maj2}, we have
$$\somsig\Biggl[\FKsig^{n+1}D(\log N-\Psi)_{K,\sigma}^{n+1}+\GKsig^{n+1}D(\log P+\Psi)_{K,\sigma}^{n+1}\Biggl]
\leq -I^{n+1}
$$
Now, using \eqref{majabs1}, \eqref{majabs2} and Young's inequality, we get 
\begin{multline*}
\Biggl\vert\FKsig^{n+1}D(\log N-\Psi)_{K,\sigma}^{D}\Biggl\vert+\Biggl\vert\GKsig^{n+1}D(\log P+\Psi)_{K,\sigma}^{D}\Biggl\vert\leq \\
\frac{\ts}{2}
\Biggl[\min(N_K^{n+1},N_{K,\sigma}^{n+1})\!\left|D_\sigma(\log N-\Psi)^{n+1}\right|^2 \!+ \frac{\max(N_K^{n+1},N_{K,\sigma}^{n+1})^2}{\min(N_K^{n+1},N_{K,\sigma}^{n+1})}\left|D_\sigma(\log N-\Psi)^{D}\right|^2 \\
+ \min(P_K^{n+1},P_{K,\sigma}^{n+1})\left|D_\sigma(\log P+\Psi)^{n+1}\right|^2+ \frac{\max(P_K^{n+1},P_{K,\sigma}^D)^2}{\min(P_K^{n+1},P_{K,\sigma}^{n+1})}\left|D_\sigma(\log P+\Psi)^{D}\right|^2.\Biggl]
\end{multline*}
Finally, thanks to the $L^{\infty}$-estimates \eqref{estLinf} in Theorem~\ref{thm-ex}, we obtain 
$$\frac{\mathbb E ^{n+1}-\mathbb E ^{n}}{\Delta t}\leq -\frac{1}{2}I^{n+1}+\frac{M^2}{2m}\, \left|\log(N_{\cal M}^D)-\Psi_{\cal M}^{D}\right|^{2}_{1,{\cal M}}+\frac{M^2}{2m}\, \left|\log(P_{\cal M}^D)+\Psi_{\mathcal M}^{D}\right|^{2}_{1,{\cal M}}$$
which rewrites
\begin{equation}\label{bientot_fini}
\ds\frac{\mathbb E ^{n+1}-\mathbb E ^{n}}{\Delta t}+\frac{1}{2}\, \mathbb I^{n+1} \leq 
\frac{M^2}{2m}\left( \left|\log(N_{\cal M}^D)-\Psi_{\cal M}^{D}\right|^{2}_{1,\T}+\left|\log(P_{\cal M}^D)+\Psi_{\cal M}^{D}\right|^{2}_{1,\T}\right).
\end{equation}
But, thanks to hypothesis \eqref{HYP-CL}, the functions $\log(N^D)-\Psi^{D}$ and $\log(P^D)+\Psi^{D}$ belong to 
$H^1(\Omega)$. Therefore, using Lemma 9.4 in \cite{EGHbook}, we have 
$$
\begin{aligned}
&\left|\log(N_{\cal M}^D)-\Psi_{\cal M}^{D}\right|^{2}_{1,{\cal M}}\leq {\cal K}\,\|\log(N^D)-\Psi^D\|^2_{H^1(\Omega)} \\
&\mbox{ and }
\left|\log(P_{\cal M}^D)+\Psi_{\cal M}^{D}\right|^{2}_{1,{\cal M}}\leq {\cal K}\,\|\log(P^D)+\Psi^D\|^2_{H^1(\Omega)}
\end{aligned}$$
with ${\cal K}$ depending on $\beta$ and $\xi$ (defined in \eqref{reg-mesh}).
It yields \eqref{entropy-estimate}.

Summing \eqref{entropy-estimate} over $n\in\{0,\ldots N_T-1\}$ yields:
\begin{equation}\label{depI}
\ds\sum_{n=0}^{N_T-1} \Delta t\, \mathbb I^{n+1}\ \leq\  \mathbb E ^{N_T}+\ds\sum_{n=0}^{N_T-1} \Delta t\, \mathbb I^{n+1}\  
\leq \ T\,K_E +\mathbb E^0.
\end{equation}
It remains now to bound $\mathbb E^0$. As the function $H$ satisfies the following inequality:
$$
\forall x,y>0 \quad H(y)-H(x)-\log x (y-x)\leq \frac{1}{\min(x,y)} \frac{(y-x)^2}{2}, 
$$
we get, using \eqref{HYP-mM},
$$
\sum_{K \in \mathcal{T}}\text{m}(K) \left(H(N^{0}_{K})-H(N^{D}_{K})-\log(N^{D}_{K})\left(N^{0}_{K}-N^{D}_{K}\right)\right)\leq \text{m}(\Omega)\frac{(M-m)^2}{2m},$$
and the same inequality for $P$.
Then, multiplying the scheme \eqref{scheme-Psi} at $n=0$ by $\Psi_K^0-\Psi_K^D$ and summing over $K\in\T$, we get 
$$
\lambda^2\somsig \ts D\Psi_{K,\sigma}^0 (D\Psi_{K,\sigma}^0-D\Psi_{K,\sigma}^D)=\sum_{K\in\T}\text{m}(K) (P_K^0-N_K^0)(\Psi_K^0-\Psi_K^D)=0,
$$
if the initial conditions satisfy the quasi-neutrality assumption \eqref{QN-CI}. Then, using 
$a(a-b)\geq (a-b)^2/2-b^2/2$ for $a,b\in\R$ and once more Lemma 9.4 in \cite{EGHbook}, we obtain
$$
\ds\frac{\lambda^2}{2}\vert \Psi_{\cal M}^0-\Psi_{\cal M}^D\vert_{1,{\cal M}}\leq \ds\frac{\lambda^2}{2}\vert\Psi_{\cal M}^D\vert\leq \ds\frac{\lambda^2}{2}{\cal K}\,\Vert\Psi^D\Vert_{H^1(\Omega)}.
$$
with ${\cal K}$ depending on $\beta$ and $\xi$ (defined in \eqref{reg-mesh}).

Finally, we obtain
$\mathbb E^0\leq K_E^0 (1+\lambda^2)$,
with $K_E^0$ depending on $\Omega$, $m$, $M$, $\Psi^D$, $\beta$ and $\xi$. Inserting this result in~\eqref{depI}, we deduce the discrete control of the entropy production \eqref{diss-entropy-estimate} with an adaptation of the constant $K_E$. 

It concludes the proof of Theorem \ref{thm-entropy}. Let us note that the hypothesis on the vanishing  doping profile is not directly necessary to follow the computations in this proof. However, we need it in order to ensure the lower and the upper bounds on the discrete densities, with their strict positivity. 

\section{A priori estimates on the scheme}\label{sec.estimates}
 This Section is devoted to the proof of Theorem \ref{thm-estap}. This proof is split into three steps: first, we establish the weak-BV inequality on $N$ and $P$ \eqref{estBV}; then, we deduce the $L^2(0,T,H^1)$-estimate on $N$ and $P$ \eqref{estL2H1_NP}; finally, we conclude with the $L^2(0,T,H^1)$-estimate on $\Psi$ \eqref{estL2H1_Psi}.

\subsection{ Weak BV-inequality on $N$ and $P$}

First, let us first prove the inequality \eqref{estBV} of Theorem \ref{thm-estap}. Therefore, we denote by $T_{BV}$ the left-hand-side of \eqref{estBV}, that is  the term we want to bound.

We follow the ideas of \cite{CCHP}: we multiply the scheme on $N$ \eqref{scheme-N} by $\Delta t\, (\NKnp-N_K^D)$ and the scheme on $P$ \eqref{scheme-P} by $\Delta t\, (\PKnp-P_K^D)$ and we sum over $K\in\T$ and $n$.  It yields 
\begin{equation}\label{depEF}
E_1+E_2+E_3+ F_1+F_2+F_3=0,
\end{equation}
with 
$$
\begin{gathered}
E_1=\!\!\!\somn\somK \!\mK(\NKnp-\NKn)(\NKnp-N_K^D),\ E_2=-\!\!\!\somn\!\! \Delta t \!\!\somsig\!\!\!\FKsig^{n+1}DN_{K,\sigma}^{n+1},\\
E_3=\somn \Delta t \!\!\!\somsig \FKsig^{n+1}DN_{K,\sigma}^{D},\ 
 F_1=\!\!\somn\somK \mK(\PKnp-\PKn)(\PKnp-P_K^D),\\
F_2=-\somn\!\!  \Delta t \somsig \GKsig^{n+1}DP_{K,\sigma}^{n+1},\ 
 F_3=\somn \Delta t \somsig \GKsig^{n+1}DP_{K,\sigma}^{D}.\\
\end{gathered}
$$

As
$(\NKnp-\NKn)(\NKnp-N_K^D)=\Bigl((\NKnp-N_K^D)^2+(\NKnp-\NKn)^2-
(\NKn-N_K^D)^2\Bigl)/2$, we get:
\begin{equation}\label{minE1F1}
\begin{aligned}
&E_1\geq -\ds\frac{1}{2} \somK \mK (N_K^0-N_K^D)^2\geq -\ds\frac{{\rm m}(\Omega)(M-m)^2}{2}
\hspace{0.5cm}\\
&\mbox{ and }\hspace{0.5cm} 
F_1\geq -\ds\frac{1}{2} \somK \mK (P_K^0-P_K^D)^2\geq -\ds\frac{{\rm m}(\Omega)(M-m)^2}{2}.
\end{aligned}
\end{equation}
We may also bound the terms $E_3$ and $F_3$. Indeed, using successively the property of the flux $\FKsig^{n+1}$ \eqref{majabs1}, the $L^{\infty}$ estimates and  Cauchy-Schwarz inequality, we get
\begin{multline*}
\vert E_3\vert \leq \somn \Delta t\somsig \ts \max (\NKnp,N_{K,\sigma}^{n+1}) D_\sigma(\log N-\Psi)^{n+1}D_\sigma N^D\\
\leq  \frac{M}{\sqrt{m}}\sqrt{T}\, \vert N_{\cal M}^D\vert_{1,{\cal M}}\Biggl( \somn \Delta t\!\!\!\somsig \!\!\!\!\ts \min (\NKnp,N_{K,\sigma}^{n+1})\left(D_\sigma(\log N-\Psi)^{n+1}\right)^2\Biggl)^{\frac{1}{2}}
\end{multline*}
But, the right-hand-side is bounded thanks to the control of the entropy production \eqref{diss-entropy-estimate} and the hypothesis \eqref{HYP-CL}. Following similar computations for $F_3$, we  get 
\begin{equation}\label{majE3}
\vert E_3\vert \leq \mathcal K(1+\lambda^2) \mbox{ and } \vert F_3\vert \leq \mathcal K (1+\lambda^2)
\end{equation}
with $\mathcal K$ depending only on $T$, $K_E$, $M$, $m$, $N^D$, $P^D$, $\beta$ and $\xi$. 

We focus now on the main terms $E_2$ and $F_2$. Using the definition of the Bernoulli function \eqref{Bern}, the numerical fluxes $\FKsig^{n+1}$ and $\GKsig^{n+1}$, defined by  \eqref{FLUX-SG}, rewrite: 
\begin{eqnarray*}
\FKsig^{n+1}&=& \frac{\ts}{2}\left[ D\Psi_{K,\sigma}^{n+1}(N_K^{n+1}+N_{K,\sigma}^{n+1})- D\Psi_{K,\sigma}^{n+1}\coth\left(\frac{D\Psi_{K,\sigma}^{n+1}}{2}\right)DN_{K,\sigma}^{n+1}\right],\\
\GKsig^{n+1}&=& \frac{\ts}{2}\left[ -D\Psi_{K,\sigma}^{n+1}(P_K^{n+1}+P_{K,\sigma}^{n+1})- D\Psi_{K,\sigma}^{n+1}\coth\left(\frac{D\Psi_{K,\sigma}^{n+1}}{2}\right)DP_{K,\sigma}^{n+1}\right].
\end{eqnarray*}
Since $x\coth(x)\geq \vert x\vert$ for all $x\in\R$, we obtain 
\begin{eqnarray*}
E_2&\geq& \frac{1}{2}\! \somn \!\!\Delta t\Biggl[\sum_{\sigma\in\E}\ts \, D_{\sigma}\Psi^{n+1}\, (D_{\sigma}N^{n+1})^2-\!\!\!\somsig\!\!\!\!\!\ts D\Psi_{K,\sigma}^{n+1}\Bigl((N_{K,\sigma}^{n+1})^2-(\NKnp)^2\Bigl)\Biggl],\\
F_2&\geq &\frac{1}{2} \!\somn \!\!\Delta t\Biggl[\sum_{\sigma\in\E}\ts \, D_{\sigma}\Psi^{n+1}\, (D_{\sigma}P^{n+1})^2+\!\!\!
\somsig\!\!\!\!\!\ts D\Psi_{K,\sigma}^{n+1}\Bigl((P_{K,\sigma}^{n+1})^2-(\PKnp)^2\Bigl)
\Biggl].
\end{eqnarray*}Summing these two inequalities, we can integrate by parts due to the quasi-neutrality of the boundary conditions \eqref{QN-CL} and we get 
$$E_2+F_2\geq\frac{1}{2}\somn\Delta t \somK\sum_{\sigma\in\E_K}\ts D\Psi_{K,\sigma}^{n+1}\Bigl((\NKnp)^2-(\PKnp)^2\Bigl)+\frac{1}{2} T_{BV}.$$
In the case $\lambda=0$, using $P^{n+1}_K=N^{n+1}_K$, we obtain
\begin{equation}\label{E2F2}
E_2+F_2\geq \frac{1}{2} T_{BV}.
\end{equation}
In the case $\lambda>0$, using the scheme on $\Psi$~\eqref{scheme-Psi}, we get:
$$E_2+F_2\geq \frac{1}{2\lambda^2}\somn\Delta t \somK\mK(\NKnp-\PKnp)\Bigl((\NKnp)^2-(\PKnp)^2\Bigl)+\frac{1}{2} T_{BV}.$$
Since the function $x\mapsto x^2$ is nondecreasing on $\R^+$, it also yields~\eqref{E2F2}. Finally, we deduce the weak-BV inequality \eqref{estBV} from \eqref{depEF}, \eqref{minE1F1}, \eqref{majE3} and 
\eqref{E2F2}.

\subsection{Discrete $L^2(0,T;H^1)$-estimates on the densities}

Now, we give the proof of  the inequality \eqref{estL2H1_NP} of Theorem \ref{thm-estap}. Therefore, we start as in the proof of  \eqref{estBV} with \eqref{depEF}. But, we treat in a different manner the terms $E_2$ and $F_2$. Indeed, for all $K\in\T$ and all $\sigma\in\E_{K}$, the Scharfetter-Gummel fluxes $\FKsig^{n+1}$ and $\GKsig^{n+1}$ defined by \eqref{FLUX-SG} rewrite 
\begin{subequations}\label{flux_tilde}
\begin{eqnarray}
\label{flux_tildeN}\!&&\hspace{-0.9cm}\FKsig^{n+1}= \ts\!\! \left ( {\tilde B}(-D\Psi_{K,\sigma}^{n+1})\NKnp\!-\!{\tilde B}(D\Psi_{K,\sigma}^{n+1})N_{K,\sigma}^{n+1}\!-\!DN_{K,\sigma}^{n+1}\right)\!\!=\! {\widetilde {\mathcal F}}_{K,\sigma}^{n+1}-\ts DN_{K,\sigma}^{n+1} \\
\label{flux_tildeP}\!&&\hspace{-0.9cm}\GKsig^{n+1}=\ts \!\! \left ({\tilde B}(D\Psi_{K,\sigma}^{n+1})\PKnp-{\tilde B}(-D\Psi_{K,\sigma}^{n+1})P_{K,\sigma}^{n+1}\!-\!DP_{K,\sigma}^{n+1}\right)\!\! =\! {\widetilde {\mathcal G}}_{K,\sigma}^{n+1}-\ts DP_{K,\sigma}^{n+1}
\end{eqnarray}
\end{subequations}
with ${\tilde B}$ defined by $ {\tilde B}(x)=B(x)-1$ for all $x\in\R$. 
Therefore 
\begin{equation}\label{E2F2bis}
E_2+F_2= \somn \Delta t\sum_{\sigma\in\E} \ts (D_\sigma N^{n+1})^2+
\somn \Delta t\sum_{\sigma\in\E} \ts (D_\sigma P^{n+1})^2 +{\tilde E}_2+{\tilde F}_2,
\end{equation}
with 
$
{\tilde E}_2= -\somn \Delta t\somsig {\widetilde {\mathcal F}}_{K,\sigma}^{n+1} DN_{K,\sigma}^{n+1}$
and $ {\tilde F}_2= -\somn \Delta t\somsig{\widetilde {\mathcal G}}_{K,\sigma}^{n+1} DP_{K,\sigma}^{n+1}.
$

But, as for the fluxes $\FKsig^{n+1}$, we can rewrite the fluxes ${\widetilde {\mathcal F}}_{K,\sigma}^{n+1}$ either under the form \eqref{fluxpos} or \eqref{fluxneg} with ${\tilde B}$ instead of $B$. Then, as $x(x-y)=\frac{1}{2}(x-y)^2+\frac{1}{2}(x^2-y^2)$, we get either 
\vspace{-0.3cm}
\begin{subequations}
\begin{eqnarray}
-{\widetilde {\mathcal F}}_{K,\sigma}^{n+1} DN_{K,\sigma}^{n+1}&=&
\ts\Biggl(\frac{D\Psi_{K,\sigma}^{n+1}}{2}(D_{\sigma}N^{n+1})^2+\frac{D\Psi_{K,\sigma}^{n+1}}{2}\left((\NKnp)^2-(N_{K,\sigma}^{n+1})^2\right)\nonumber\\
[-8pt]&&\hspace{4.2cm}+\tilde B (D\Psi_{K,\sigma}^{n+1})(D_{\sigma}N^{n+1})^2\Biggl),\label{relneg}\\
\mbox{ or } -{\widetilde {\mathcal F}}_{K,\sigma}^{n+1} DN_{K,\sigma}^{n+1}&=&
\ts\Biggl(-\frac{D\Psi_{K,\sigma}^{n+1}}{2}(D_{\sigma}N^{n+1})^2\!-\frac{D\Psi_{K,\sigma}^{n+1}}{2}\left((N_{K,\sigma}^{n+1})^2-(\NKnp)^2\right)\!\nonumber\\
[-8pt]&&\hspace{4.cm}+\!\tilde B (-D\Psi_{K,\sigma}^{n+1})(D_{\sigma}N^{n+1})^2\Biggl).\label{relpos}
\end{eqnarray}
\end{subequations}
But, $\tilde B(x)\geq 0$ for all $x\leq 0$ and $\tilde B (-x)\geq 0$ for all $x\geq 0$. Then, using \eqref{relneg} when $D\Psi_{K,\sigma}^{n+1}\leq 0$ and \eqref{relpos} when $D\Psi_{K,\sigma}^{n+1}\geq 0$, we obtain in both cases
$$
-{\widetilde {\mathcal F}}_{K,\sigma}^{n+1} DN_{K,\sigma}^{n+1} \geq 
\ds\frac{\ts}{2} \left (-D_{\sigma}\Psi^{n+1}(D_{\sigma}N^{n+1})^2+ D\Psi_{K,\sigma}^{n+1}\left((\NKnp)^2-(N_{K,\sigma}^{n+1})^2\right)\right).
$$
Similarly, we have 
$$
-{\widetilde {\mathcal G}}_{K,\sigma}^{n+1} DP_{K,\sigma}^{n+1} \geq 
\ds\frac{\ts}{2} \left (-D_{\sigma}\Psi^{n+1}(D_{\sigma}P^{n+1})^2- D\Psi_{K,\sigma}^{n+1}\left((\PKnp)^2-(P_{K,\sigma}^{n+1})^2\right)\right).
$$
It yields, after a discrete integration by parts
$$
\tilde E_2+\tilde F_2\geq -\ds\frac{1}{2}T_{BV}+\frac{1}{2}\somn\Delta t\somK\sum_{\sigma\in\E_K} \ts D\Psi_{K,\sigma}^{n+1} \left((\NKnp)^2-(\PKnp)^2\right),$$
and, thanks to the scheme \eqref{scheme-Psi},
\begin{equation}\label{Et2Ft2}
\tilde E_2+\tilde F_2\geq -\ds\frac{1}{2}T_{BV}.
\end{equation}
Then, we deduce the discrete $L^2(0,T;H^1)$ estimate on $N$ and $P$ \eqref{estL2H1_NP} from \eqref{depEF}, \eqref{minE1F1}, \eqref{majE3}, \eqref{E2F2bis}, \eqref{Et2Ft2} and the weak-BV inequality \eqref{estBV}. 

\subsection{Discrete $L^2(0,T;H^1(\Omega))$ estimate on $\Psi$}

We conclude the proof of Theorem \ref{thm-estap} with the proof of the $L^2(0,T,H^1)$ estimate on $\Psi$ \eqref{estL2H1_Psi}. 
Once more, we use Theorem \ref{thm-entropy} in the proof. 

Let us first consider the case $\lambda=0$. In this case, multiplying the scheme on $\Psi$ \eqref{Sz-Psi} by $\Delta t(\Psi_K^{n+1}-\Psi_K^D)$ and summing over $K\in\T$ and $n\in \{0,\ldots,N_T-1\}$, we get:
$$
\somn \Delta t\somsig \ts D\Psi_{K,\sigma}^{n+1}\left(D\Psi_{K,\sigma}^{n+1}-D\Psi_{K,\sigma}^{D}\right)\left(N_K^{n+1}+N_{K,\sigma}^{n+1}\right)=0. 
$$
Then, thanks to the $L^{\infty}$-estimate \eqref{estLinf} and the inequality $a(a-b)\geq a^2/2-b^2/2$ ($\forall a,b \in\R$), we obtain:
$$
\somn \Delta t\somsige \ts (D_\sigma\Psi^{n+1})^2\leq \somn \Delta t\somsige \ts (D_\sigma\Psi^{D})^2,
$$
which yields \eqref{estL2H1_Psi}. 

Now, let us consider the case $\lambda>0$. We follow the ideas developed by I. Gasser in \cite{Gasser} at the continuous level and adapt them to the case of mixed  boundary conditions. 
We set:
$$
\begin{aligned}
&{\mathcal J}= \somn\Delta t\somK \mK \ds\frac{(N_K^{n+1}-P_K^{n+1})^2}{\lambda^2}\\
&\hspace{1cm}+\somn \Delta t\somsig \ts (\min (\NKnp,N_{K,\sigma}^{n+1})+ \min (\PKnp,P_{K,\sigma}^{n+1}))(D_\sigma\Psi^{n+1})^2.
\end{aligned}
$$
Multipliying the scheme on $\Psi$ \eqref{scheme-Psi} by  $\Delta t(\PKnp-\NKnp)/\lambda^2$ and summing over $K\in\T$ and $n\in \{0,\ldots,N_T-1\}$, we get 
$$
\somn\Delta t\somK \mK \ds\frac{(N_K^{n+1}-P_K^{n+1})^2}{\lambda^2}=\somn\Delta t\somsig\ts D\Psi_{K,\sigma}^{n+1}(DP_{K,\sigma}^{n+1}-DN_{K,\sigma}^{n+1}),
$$
due to the quasi-neutrality of the boundary conditions \eqref{QN-CL}. Therefore, ${\mathcal J}$ may be split into ${\mathcal J}={\mathcal J}_1+{\mathcal J}_2$ with 
\begin{eqnarray*}
{\mathcal J}_1&=& \somn\Delta t\somsig \ts D\Psi_{K,\sigma}^{n+1}\Biggl(\min (\PKnp,P_{K,\sigma}^{n+1}) D(\log P+\Psi)_{K,\sigma}^{n+1}\\
[-20pt]&&\hspace{5.8cm}-\min (\NKnp,N_{K,\sigma}^{n+1}) D(\log N-\Psi)_{K,\sigma}^{n+1}\Biggl),\\
[-8pt]{\mathcal J}_2&=&\somn\Delta t\somsig \ts D\Psi_{K,\sigma}^{n+1}\Biggl(\left(DP_{K,\sigma}^{n+1}-\min(\PKnp,P_{K,\sigma}^{n+1}) D(\log P)_{K,\sigma}^{n+1}\right)\\
[-20pt]&&\hspace{4.6cm}
-\left(DN_{K,\sigma}^{n+1}-\min(\NKnp,N_{K,\sigma}^{n+1}) D(\log N)_{K,\sigma}^{n+1}\right)\Biggl).
\end{eqnarray*}

Applying Young inequality on ${\mathcal J}_1$, we get 
\begin{multline*}
\hspace{-0.5cm}\left\vert {\mathcal J}_1\right\vert \leq \frac{1}{2}\!\! \somn\!\!\Delta t\!\Biggl[\somsig \!\!\ts (D_\sigma\Psi^{n+1})^2\!\left(\min (\NKnp,N_{K,\sigma}^{n+1})+ \min (\PKnp,P_{K,\sigma}^{n+1})\right)\!+\! \mathbb I^{n+1}\Biggl]
\\
\hspace{-0.2cm}\leq \frac{1}{2} \!\somn\!\!\Delta t\!\!\somsig \!\ts (D_\sigma\Psi^{n+1})^2\left(\min (\NKnp,N_{K,\sigma}^{n+1})
+ \min (\PKnp,P_{K,\sigma}^{n+1})\right)\\
+\frac{K_E(1+\lambda^2)}{2}.
\end{multline*}

Now, we estimate the term ${\mathcal J}_2$ which does not appear at the continuous level because $\nabla N=N\nabla \log N$. For all $x,y>0$ we have 
$$
\left\vert \log y-\log x - \ds\frac{y-x}{\min(x,y)}\right\vert \leq \frac{(x-y)^2}{2\min(x,y)^2}.
$$
It yields 
$$
\begin{gathered}
\left\vert DP_{K,\sigma}^{n+1}-\min(\PKnp,P_{K,\sigma}^{n+1}) D(\log P)_{K,\sigma}^{n+1}\right\vert \leq \frac{(DP_{K,\sigma}^{n+1})^2}{2m},\\
\left\vert DN_{K,\sigma}^{n+1}-\min(\NKnp,N_{K,\sigma}^{n+1}) D(\log N)_{K,\sigma}^{n+1}\right\vert \leq \frac{(DN_{K,\sigma}^{n+1})^2}{2m}
\end{gathered}
$$
and 
$$
\left\vert {\mathcal J}_2\right\vert \leq \frac{1}{2m}\somn\Delta t\somsigint \ts\vert D\Psi_{K,\sigma}^{n+1}\vert 
\Bigl((DP_{K,\sigma}^{n+1})^2+ (DN_{K,\sigma}^{n+1})^2\Bigl)
\leq  \frac{K_{BV}(1+\lambda^2)}{2m} .
$$

As ${\mathcal J}={\mathcal J}_1+{\mathcal J}_2$, the estimates on ${\mathcal J}_1$ and ${\mathcal J}_2$ imply that 
$$
\mathcal J\leq \frac{mK_E+K_{BV}}{2m}(1+\lambda^2).$$
As $N$ and $P$ are lower bounded by $m$ \eqref{estLinf}, it yields \eqref{estL2H1_Psi} in the case $\lambda>0$.

\section{Numerical experiments}\label{sec.numexp}

In this section we present some numerical results in one and two space dimensions. Our purpose is to illustrate the stability of the fully implicit Scharfetter-Gummel scheme for all nonnegative values of the rescaled Debye length $\lambda$. 

\subsection{Test case 1: 1D with $C=0$}

First, we consider a one dimensional test case on $\Omega=(0,1)$, with a zero doping profile since this situation corresponds to the one studied in this paper. Initial data are constant
$N_{0}(x)=P_{0}(x)=0.5$, $\forall x \in (0,1)$.
We consider quasi-neutral Dirichlet boundary conditions
$N^{D}(0)=P^{D}(0)=0.1$, $\Psi^{D}(0)=0$ and $N^{D}(1)=P^{D}(1)=0.9$, $\Psi^{D}(1)=4$.

\begin{figure}[ht!]
\psfrag{lambda1}{\scriptsize$\lambda^2=1$}
\psfrag{lambda2}{\scriptsize$\lambda^2=10^{-2}$}
\psfrag{lambda3}{\scriptsize$\lambda^2=10^{-5}$}
\psfrag{lambda4}{\scriptsize$\lambda^2=10^{-9}$}
\psfrag{lambda5}{\scriptsize$\lambda^2=0$}
\psfrag{deltat}{$\Delta t$}
\psfrag{Nerror}{$L^1$ error on $N$}
\psfrag{Psierror}{$L^1$ error on $\Psi$}
\includegraphics[scale=0.33]{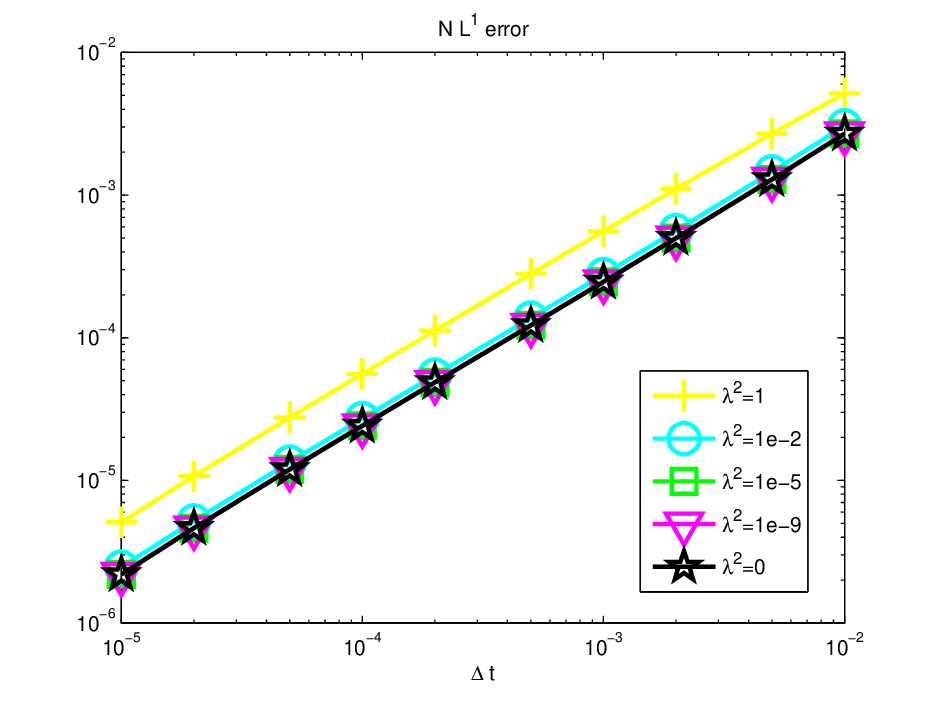}\hspace{-0.2cm}
\includegraphics[scale=0.33]{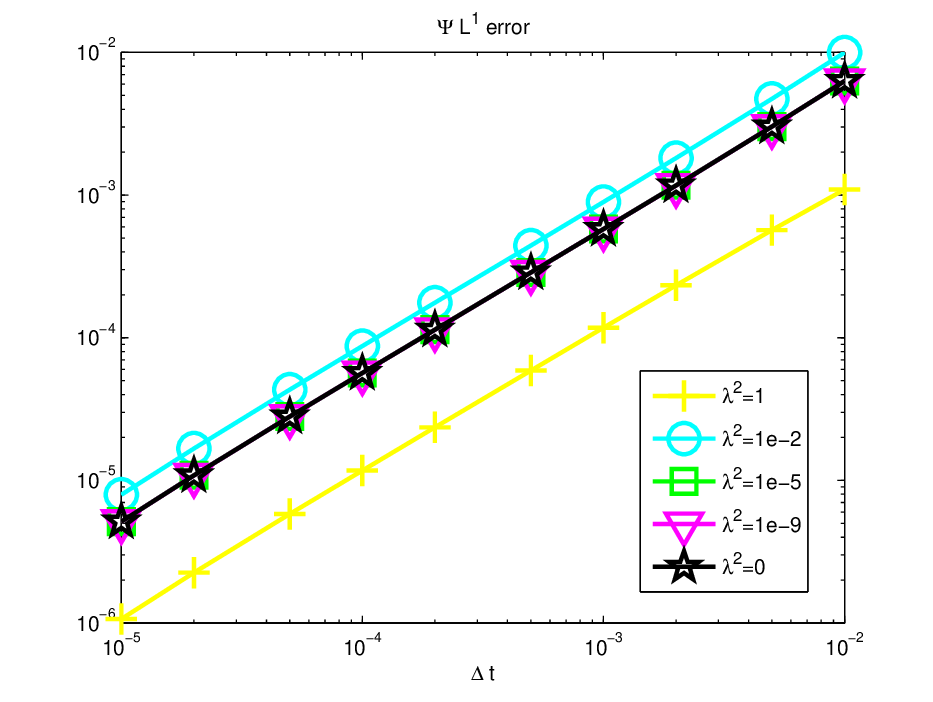}
\caption{Test case 1. Errors in $L^1$ norm as functions of $\Delta t$, for different values of $\lambda^{2}$ .}
\label{test1_courbe2}
\end{figure}

\begin{figure}[ht!]
\psfrag{dx1}{\scriptsize$\Delta x=20^{-1}$}
\psfrag{dx2}{\scriptsize$\Delta x=40^{-1}$}
\psfrag{dx3}{\scriptsize$\Delta x=80^{-1}$}
\psfrag{dx4}{\scriptsize$\Delta x=160^{-1}$}
\psfrag{dx5}{\scriptsize$\Delta x=320^{-1}$}
\psfrag{dx6}{\scriptsize$\Delta x=640^{-1}$}
\psfrag{dx7}{\scriptsize$\Delta x=1280^{-1}$}
\psfrag{dx8}{\scriptsize$\Delta x=2560^{-1}$}
\psfrag{lambda2}{$\lambda^2$}
\psfrag{Nerror}{$L^1$ error on $N$}
\psfrag{Psierror}{$L^1$ error on $\Psi$}
\includegraphics[scale=0.35]{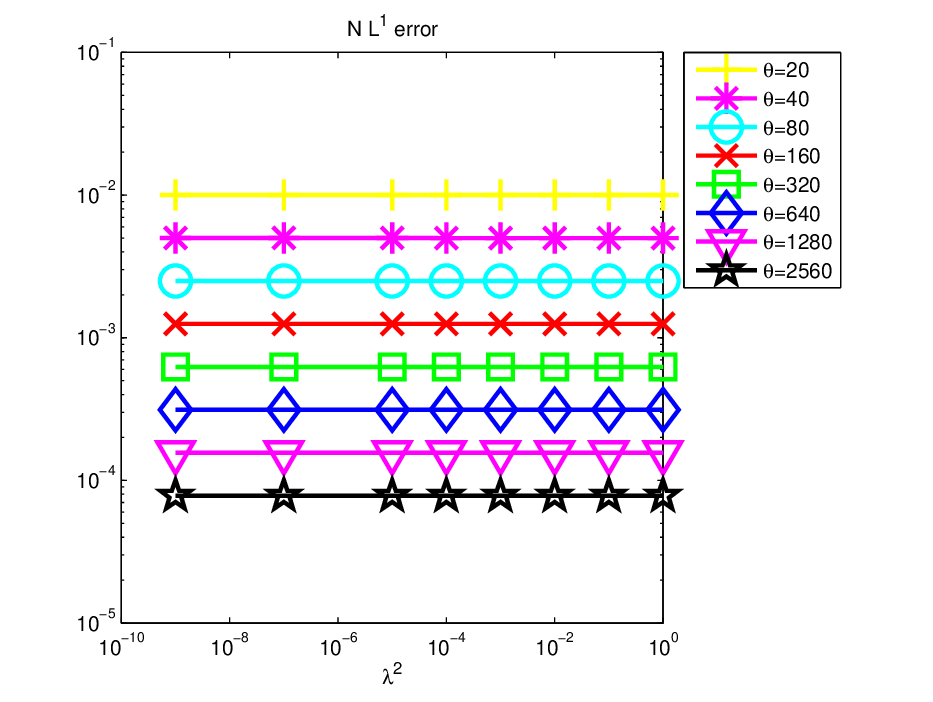}\hspace{-0.2cm}
\includegraphics[scale=0.35]{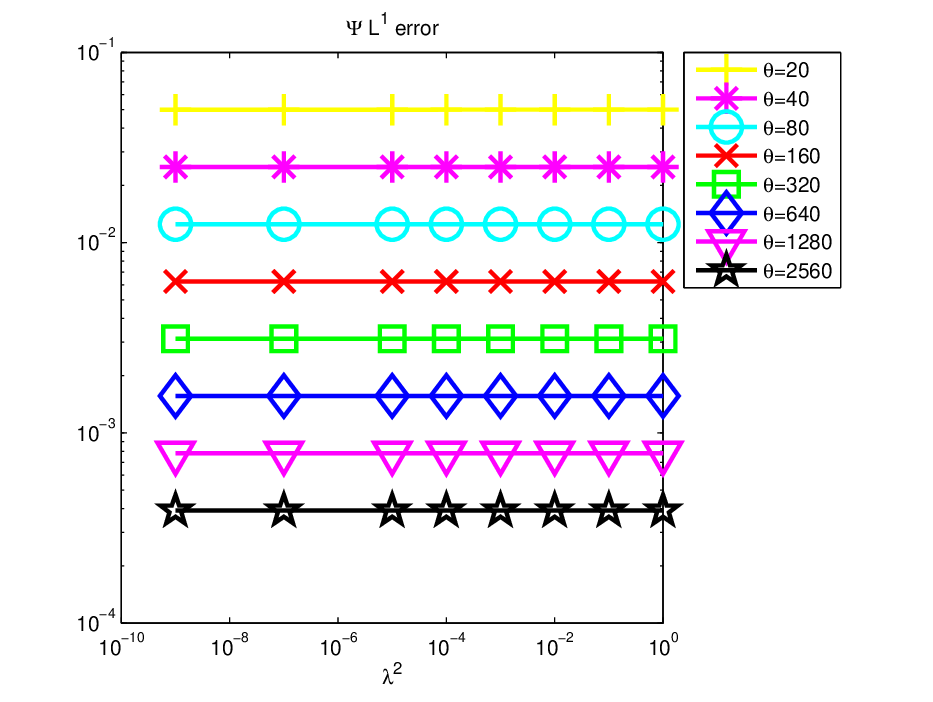}
\caption{Test case 1. Errors in $L^1$ norm  as functions of $\lambda^{2}$, for different values of $\Delta x$ .}
\label{test1_courbe4}
\end{figure}

Since the exact solution of this problem is not available, we compute a reference solution on a uniform mesh made of $10240=20\times 2^9$ cells, with time 
step $\Delta t=10^{-6}$, for different values of $\lambda^{2}$ in $[0,1]$. This reference solution is then used to compute the $L^{1}$ error
 for the variables $N$, $P$ and $\Psi$. In order to prove the asymptotic preserving behavior of the scheme, we compute $L^{1}$ errors at time $T=0.1$ for different numbers of cells $\theta=20\times 2^{i}$, $i\in \{0,...,8\}$, with different time steps $\Delta t$ in $[10^{-5},10^{-2}]$ and various rescaled Debye length $\lambda^{2}$ in $[0,1]$. Figure \ref{test1_courbe2} presents the $L^1$ error on the electron density and on the electrostatic potential as functions of $\Delta t$ for different values of $\lambda^{2}$. It clearly shows the uniform behavior in the limit $\lambda$ tends to $0$ since the convergence rate is of order 1 for all variables even for small values of $\lambda^{2}$, including zero. Similar results are obtained for the hole density.

We plot the $L^1$ errors as functions of $\lambda^2$ for different values of the space step on Figure~\ref{test1_courbe4}. We still observe the asymptotic preserving properties of the scheme in the limit $\lambda$ tends to zero. Moreover, the errors are independent of $\lambda^{2}$. 

\subsection{Test case 2: 1D with a discontinuous doping profile}

Here, we consider a nonzero discontinuous doping profile on $\Omega=(0,1)$, which corresponds to the physically relevant hypothesis, but not to the framework of our study:
\begin{equation*}
C(x)=\left\{\begin{array}{lcl}
-0.8 & \text{ for } & x \leq 0.5, \\ +0.8 &\text{ for } & x>0.5.
\end{array}\right.
\end{equation*}
The initial conditions are
$N_{0}(x)=(1+C(x))/2$, $P_{0}(x)=(1-C(x))/2$ for all $x \in (0,1)$.
And, the boundary conditions are still quasi-neutral and of the Dirichlet type
$N^{D}(0)=0.1$,\quad $P^{D}(0)=0.9$, $\Psi^{D}(0)=0$ and
 $N^{D}(1)=0.9$, $P^{D}(1)=0.1$, $\Psi^{D}(1)=4$.
\begin{figure}[ht!]
\psfrag{lambda1}{\scriptsize$\lambda^2=1$}
\psfrag{lambda2}{\scriptsize$\lambda^2=10^{-3}$}
\psfrag{lambda3}{\scriptsize$\lambda^2=10^{-7}$}
\psfrag{lambda4}{\scriptsize$\lambda^2=0$}
\psfrag{deltat}{$\Delta t$}
\psfrag{Nerror}{$L^1$ error on $N$}
\psfrag{Psierror}{$L^1$ error on $\Psi$}
\subfigure{\includegraphics[scale=0.32]{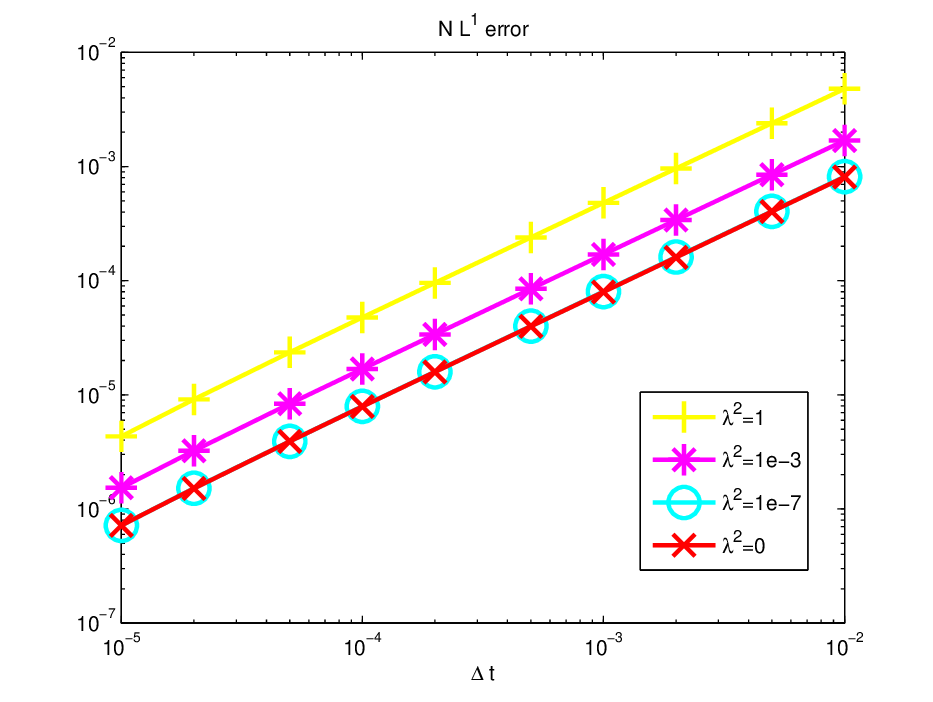}}
\subfigure{\includegraphics[scale=0.32]{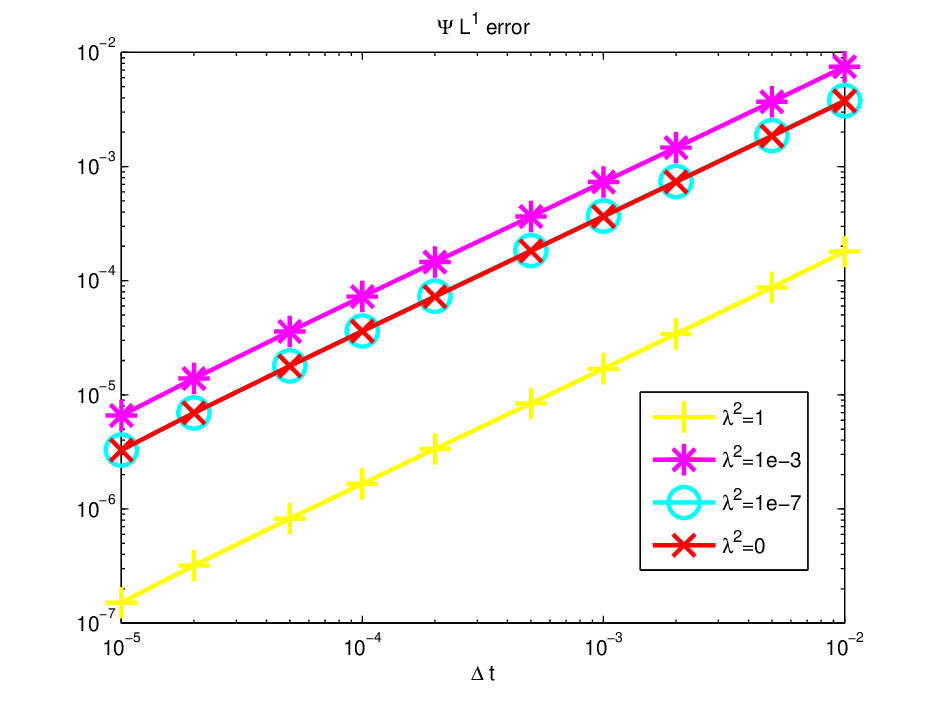}}
\caption{Test case 2. Errors in $L^1$ norm  as functions of $\Delta t$, for different values of $\lambda^{2}$.}
\label{test2_courbe1}
\end{figure}
Figure \ref{test2_courbe1} presents the error in $L^{1}$ norm between the approximate solution and the reference solution computed as previously. We observe that  the convergence rate does not depend on the value of $\lambda^{2}$. It seems that the scheme is still asymptotic preserving at the quasi-neutral limit even if the doping profile $C$ is not zero.

\subsection{Test case 3: PN-junction in 2D}

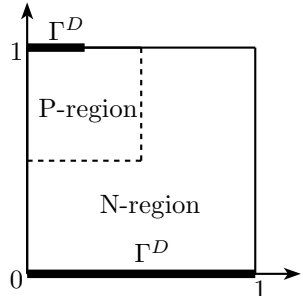
\begin{figure}[ht!]
\psset{xunit=0.3cm,yunit=0.3cm,algebraic=true,dotstyle=o,dotsize=3pt 0,linewidth=0.8pt,arrowsize=3pt 2,arrowinset=0.25}
\centerline{
\begin{pspicture*}(-4.11,-7.94)(14.21,6.1)
\psline{->}(-1,-7)(-1,5)
\psline{->}(-1,-7)(11,-7)
\psline[linewidth=3.2pt](-1,-7)(9,-7)
\psline(9,-7)(9,3)
\psline(9,3)(-1,3)
\psline(-1,3)(-1,-7)
\psline[linestyle=dashed,dash=2pt 2pt](-1,-2)(4,-2)
\psline[linestyle=dashed,dash=2pt 2pt](4,-2)(4,3)
\psline(4,3)(-1,3)
\psline(-1,3)(-1,-2)
\rput[tl](-0.5,0.74){P-region}
\rput[tl](2.18,-3.5){N-region}
\psline[linewidth=3.2pt](-1,3)(1.52,3)
\rput[tl](-0.06,4.33){$\Gamma^{D}$}
\rput[tl](3.7,-5.50){$\Gamma^{D}$}
\rput[tl](-1.76,-6.88){0}
\rput[tl](8.92,-7.32){1}
\rput[tl](-1.76,3.21){1}
\end{pspicture*}}
\caption{Geometry of the PN-junction diode}
\label{jonctionPN}
\end{figure}

Now, we present a test case for a geometry corresponding to a PN-junction in 2D (see Figure \ref{jonctionPN}). 
The domain $\Omega$ is the square $(0,1)^{2}$. The doping profile is piecewise constant, equal to $0.8$ in the N-region and $-0.8$ in the P-region. The Dirichlet boundary conditions are
$N^{D}=0.9$, $P^{D}=0.1$, $\Psi^{D}=1.1$ on $\{y=0\}$, and
$N^{D}=0.1$, $P^{D}=0.9$, $\Psi^{D}=-1.1$ on $\{y=1, \,0\leq x\leq 0.25\}$.
Elsewhere we put homogeneous Neumann boundary conditions. Initial conditions are
$N_{0}(x,y)=(1+C(x,y))/2$, and $P_{0}(x,y)=(1-C(x,y))/2$ .

The electron density profile at time $T=1$ with a mesh made of 3584 triangles and a time step $\Delta t=10^{-2}$ for $\lambda^{2}=1$ and $10^{-9}$ are shown in Figure~\ref{test3_courbe1}. We observe that the scheme remains efficient even for small values of the Debye length and with a large time step.

\begin{figure}[ht!]
\subfigure[Electron density $N$, $\lambda^{2}=1$.]{\includegraphics[scale=0.4]{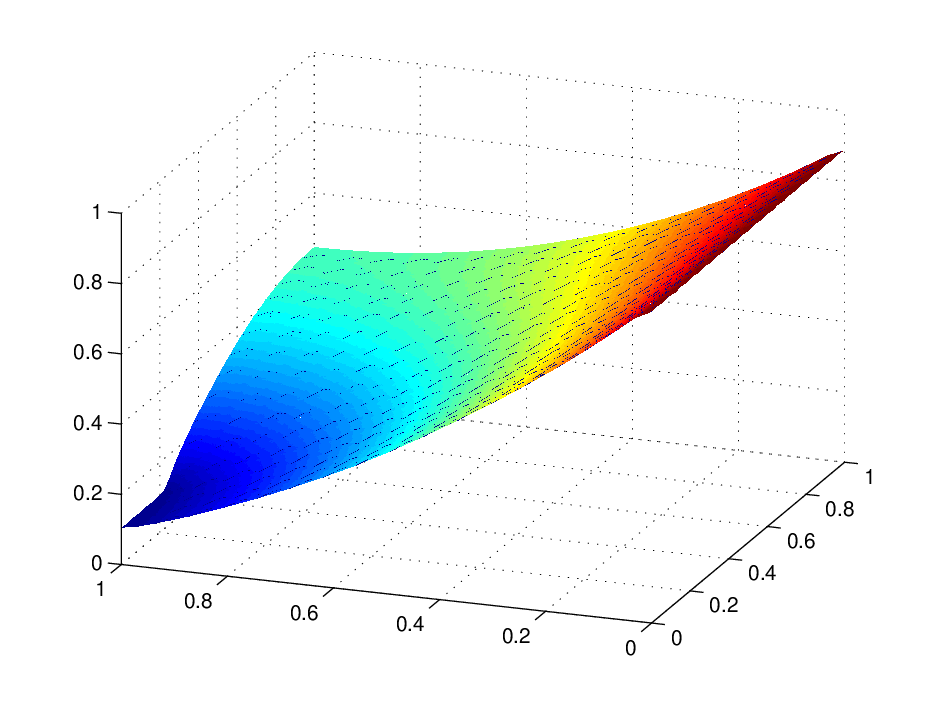}}
\subfigure[Electron density $N$, $\lambda^{2}=10^{-9}$.]{\includegraphics[scale=0.4]{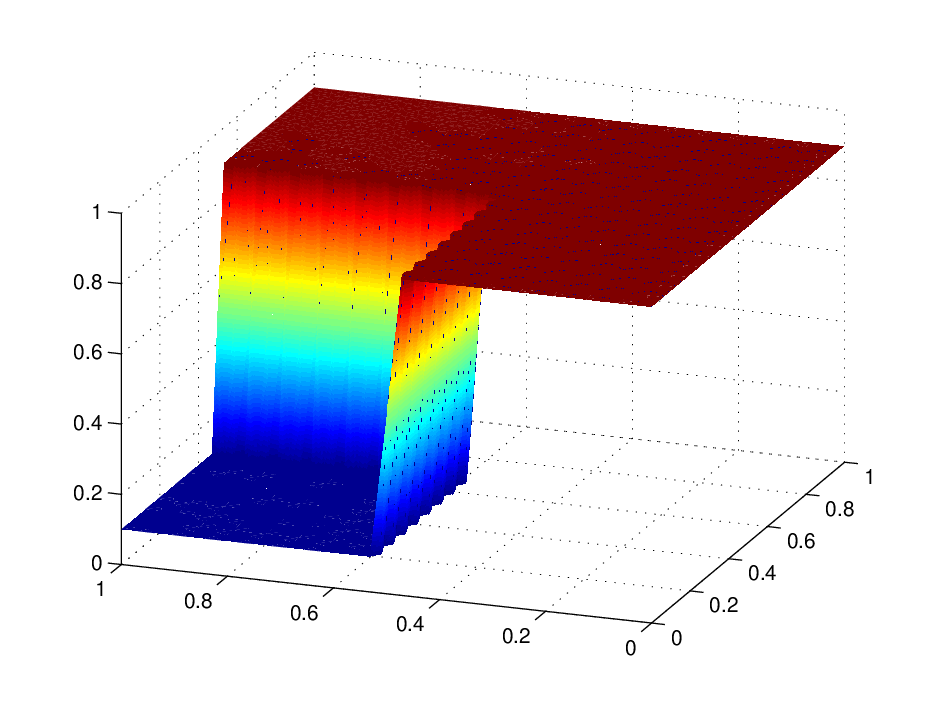}}
\caption{Test case 3. Electron density computed at time $T=1$  with a mesh of 3584 triangles and a time step $\Delta t=10^{-2}$   for $\lambda^{2}=1$ and  $\lambda^{2}=10^{-9}$.}
\label{test3_courbe1}
\end{figure}

\appendix
\section{Properties of the Scharfetter-Gummel numerical fluxes}

We recall that the Scharfetter-Gummel numerical fluxes $\FKsig^{n+1}$ and $\GKsig^{n+1}$ defined by \eqref{FLUX-SG}
can be seen respectively as numerical approximations of 
$\int_\sigma (-\nabla N+N\nabla \Psi)\cdot \nKsig$  and $\int_\sigma (-\nabla P-P\nabla \Psi)\cdot \nKsig$ on the interval $[t^n, t^{n+1})$. At the continuous level, we may rewrite $-\nabla N+N\nabla \Psi=-N\nabla (\log N -\Psi)$ and $-\nabla P-P \nabla \Psi=-P\nabla (\log P+\Psi)$. Such  equalities cannot be kept at the discrete level. However, we can give lower and upper bounds of $\FKsig^{n+1}$ and $\GKsig^{n+1}$  by terms of the form $-N_{\sigma}^{n+1} D(\log N-\Psi)_{K,\sigma}^{n+1}$ and $-P_{\sigma}^{n+1} D(\log P+\Psi)_{K,\sigma}^{n+1}$ , as shown in 
Proposition~\ref{fluxprop}. 
\begin{proposition}\label{fluxprop}
For all $K\in\T$ and all $\sigma\in\E_K$, the flux $\FKsig^{n+1}$ defined by~\eqref{FLUX-SG-N} satisfies the following inequalities:
\begin{subequations}\label{fluxpropF}
\begin{multline}
\mbox{If } D(\log N-\Psi)_{K,\sigma}^{n+1}\geq 0, \quad -\max(\NKnp,N_{K,\sigma}^{n+1})D(\log N-\Psi)_{K,\sigma}^{n+1}\leq\frac{\FKsig^{n+1}}{\ts}\\
\mbox{ and }\frac{\FKsig^{n+1}}{\ts} \leq-\min(\NKnp,N_{K,\sigma}^{n+1})D(\log N-\Psi)_{K,\sigma}^{n+1}
.\label{fluxprop1}
\end{multline}
\begin{multline}
\mbox{If } D(\log N-\Psi)_{K,\sigma}^{n+1}\leq 0,\quad -\min(\NKnp,N_{K,\sigma}^{n+1})D(\log N-\Psi)_{K,\sigma}^{n+1}\leq\frac{\FKsig^{n+1}}{\ts}\\
\mbox{ and } \frac{\FKsig^{n+1}}{\ts}\leq  -\max(\NKnp,N_{K,\sigma}^{n+1})D(\log N-\Psi)_{K,\sigma}^{n+1}
.\label{fluxprop2}
\end{multline}
\end{subequations}

Replacing $\Psi$ by $-\Psi$ and $N$ by $P$ yields similar properties for the flux $\GKsig^{n+1}$.
\end{proposition}
\begin{proof}
Let $K\in\T$, first, we remark that \eqref{fluxpropF} is trivially satisfied if $\sigma\in\E_{K,ext}^N$ because all the terms of the inequalities vanish. 
Let $\sigma\in\E_{K,int}\cup\E_{K,ext}^D$, since the Bernoulli function $B$ defined by \eqref{Bern} satisfies the property \eqref{propBern},
 the flux $\FKsig^{n+1}$ defined by \eqref{FLUX-SG-N} can be either rewritten 
 \begin{subequations}\label{fluxposneg}
\begin{equation}\label{fluxpos}
\mathcal{F}_{K,\sigma}^{n+1}=\tau_{\sigma}\biggl(D\Psi_{K,\sigma}^{n+1}N_{K}^{n+1}-B\left(D\Psi_{K,\sigma}^{n+1}\right)DN_{K,\sigma}^{n+1}\biggl),
\end{equation}
\vspace{-0.5cm}
\begin{equation}\label{fluxneg}
\mbox{ or }\quad 
\mathcal{F}_{K,\sigma}^{n+1}=\tau_{\sigma}\biggl(D\Psi_{K,\sigma}^{n+1}N_{K,\sigma}^{n+1}-B\left(-D\Psi_{K,\sigma}^{n+1}\right)DN_{K,\sigma}^{n+1}\biggl).
\end{equation}
\end{subequations}
It implies
$$
\begin{aligned}
&\mathcal{F}_{K,\sigma}^{n+1}=\tau_{\sigma}\biggl[D\Psi_{K,\sigma}^{n+1}N_{K}^{n+1}-B\left(D(\log N)_{K,\sigma}^{n+1}\right)DN_{K,\sigma}^{n+1} \hspace{6cm}\\
&\hspace{4.5cm}+\biggl(B\left(D(\log N)_{K,\sigma}^{n+1}\right)-B\left(D\Psi_{K,\sigma}^{n+1}\right)\biggl)DN_{K,\sigma}^{n+1}\biggl],\\
\textrm{and } & \mathcal{F}_{K,\sigma}^{n+1}=\tau_{\sigma}\biggl[D\Psi_{K,\sigma}^{n+1}N_{K,\sigma}^{n+1}-B\left(-D(\log N)_{K,\sigma}^{n+1}\right)DN_{K,\sigma}^{n+1} \hspace{6cm}\\
&\hspace{4.cm}+\biggl(B\left(-D(\log N)_{K,\sigma}^{n+1}\right)-B\left(-D\Psi_{K,\sigma}^{n+1}\right)\biggl)DN_{K,\sigma}^{n+1}\biggl].
\end{aligned}
$$
But, the definition of the Bernoulli function \eqref{Bern} also ensures that 
$$
B(\log y -\log x)=\ds\frac{\log y -\log x}{y-x} x,\quad \forall x,y>0.
$$
Therefore, we get 
$$\mathcal{F}_{K,\sigma}^{n+1}=\tau_{\sigma}\!\left[-D(\log N-\Psi)_{K,\sigma}^{n+1} \NKnp+\!\!   \left(B\left(D(\log N)_{K,\sigma}^{n+1}\right)\!\!-\!\!B\left(D\Psi_{K,\sigma}^{n+1}\right)\right)DN_{K,\sigma}^{n+1}\right],$$
\textrm{ and } 
$$\mathcal{F}_{K,\sigma}^{n+1}=\tau_{\sigma}\!\left[  -D(\log N-\Psi)_{K,\sigma}^{n+1} N_{K,\sigma}^{n+1}+\!\!   \left(B\left(-D(\log N)_{K,\sigma}^{n+1}\right)\!\!-\!\!B\left(-D\Psi_{K,\sigma}^{n+1}\right)\right)DN_{K,\sigma}^{n+1} \right]
$$
Now, we may use the fact that $B$ is a non increasing function on $\R$. Assuming that the sign of 
$D(\log N-\Psi)_{K,\sigma}^{n+1}$ is known, the sign of $\left(B\left(D(\log N)_{K,\sigma}^{n+1}\right)-B\left(D\Psi_{K,\sigma}^{n+1}\right)\right)$ and $\left(B\left(-D(\log N)_{K,\sigma}^{n+1}\right)-B\left(-D\Psi_{K,\sigma}^{n+1}\right)\right)$ are also known (and opposite). Distinguishing the cases $DN_{K,\sigma}^{n+1}\geq 0$ ($N_K^{n+1}\leq N_{K,\sigma}^{n+1}$) and  
$DN_{K,\sigma}^{n+1}\leq 0$ ($N_K^{n+1}\geq N_{K,\sigma}^{n+1}$) yields inequalities~\eqref{fluxpropF}.
\end{proof}

Now, we give a straightforward consequence of Proposition \ref{fluxprop} as a Corollary.
\begin{corollary}
For all $K\in\T$ and all $\sigma\in\E_K$, the fluxes $\FKsig^{n+1}$ and $\GKsig^{n+1}$ defined  by \eqref{FLUX-SG} verify:
\begin{subequations}
\begin{align}
\FKsig^{n+1}\  D(\log N-\Psi)_{K,\sigma}^{n+1} \ \leq &-\ts \min(\NKnp,N_{K,\sigma}^{n+1})\left( D_\sigma(\log N-\Psi)^{n+1}\right)^2,\label{maj1}\\
\GKsig^{n+1}\  D(\log P+\Psi)_{K,\sigma}^{n+1} \ \leq &-\ts \min(\PKnp,P_{K,\sigma}^{n+1})\left( D_\sigma(\log P+\Psi)^{n+1}\right)^2.\label{maj2}
\end{align}
\end{subequations}
Moreover, if $\min(\NKnp,N_{K,\sigma}^{n+1})\geq 0$ and $\min(\PKnp,P_{K,\sigma}^{n+1})\geq 0$, we also have 
\begin{subequations}
\begin{align}
\left\vert \FKsig^{n+1}\right\vert&\leq \ts \max(\NKnp,N_{K,\sigma}^{n+1})\left\vert D_\sigma\left(\log N-\Psi\right)^{n+1}\right\vert,\label{majabs1}\\
\left\vert \GKsig^{n+1}\right\vert&\leq \ts \max(\PKnp,P_{K,\sigma}^{n+1})\left\vert D_\sigma\left(\log P+\Psi\right)^{n+1}\right\vert.\label{majabs2}
\end{align}
\end{subequations}
\end{corollary}

\end{document}